\documentclass[11pt, reqno]{amsart} \textwidth=14.9cm \oddsidemargin=1cm \evensidemargin=1cm
\usepackage{amsmath,amstext,amsbsy,amssymb,amscd}
\usepackage{amsmath}
\usepackage{amsxtra}
\usepackage{amscd}
\usepackage{amsthm}
\usepackage{amsfonts}
\usepackage{amssymb}
\usepackage{eucal}
\usepackage{color}
\usepackage{graphicx}%
\newtheorem{thm}{Theorem}[section]
\theoremstyle{plain}

\newtheorem{lem}[thm]{Lemma}
\newtheorem{prop}[thm]{Proposition}
\newtheorem{cor}[thm]{Corollary}

\theoremstyle{definition}
\newtheorem{defn}[thm]{Definition}

\theoremstyle{remark}
\newtheorem{rem}[thm]{Remark}

\definecolor{A}{rgb}{.75,1,.75}

\numberwithin{equation}{section}

\newcommand{\C}{\mathbb C}
\newcommand{\Z}{\mathbb Z}

\newcommand{\F}{\mathbb F}

\newcommand{\al}{\alpha}
\newcommand{\Cnr}{\mathcal{C}_n^r}

\newcommand{\n}{\widehat{n}}

\newcommand{\Hn}{\mathcal H_n}
\newcommand{\Hno}{{}^0{\mathcal H}_n}
\newcommand{\Hng}{\mathcal H_n (G)}
\newcommand{\HngModI}{\mathcal{H}_n(G)\text{-\bf mod}{}_{\mathbb{I}}}
\newcommand{\HngMods}{\mathcal{H}_n(G)\text{-\bf mod}^s}
\newcommand{\HnghatMods}{\mathcal{H}_{\n}(G)\text{-\bf mod}^s}

{\vskip-\lastskip\medskip
  \noindent
  {\em #1.}\enspace
  }%
{\qed\par\medskip
  }

\begin{document}

\title[Wreath Hecke algebras]{Modular representations and branching rules for wreath Hecke algebras}
\author{Jinkui Wan and Weiqiang Wang}\thanks{Partially supported by NSF grant DMS--0800280.}
\address{Department of Mathematics, University of Virginia,
Charlottesville,VA 22904, USA.}\email{jw5ar@virginia.edu (Wan),
ww9c@virginia.edu (Wang)}

\begin{abstract}
We introduce a generalization of degenerate affine Hecke algebra,
called wreath Hecke algebra, associated to an arbitrary finite
group $G$. The simple modules of the wreath Hecke algebra and of
its associated cyclotomic algebras are classified over an
algebraically closed field of any characteristic $p\geq 0$. The
modular branching rules for these algebras are obtained, and when
$p$ does not divide the order of $G$, they are further identified
with crystal graphs of integrable modules for quantum affine
algebras. The key is to establish an equivalence between a module
category of the (cyclotomic) wreath Hecke algebra and its suitable
counterpart for the degenerate affine Hecke algebra.
\end{abstract}

\maketitle

 \setcounter{tocdepth}{1}
\tableofcontents

\section{Introduction}
\subsection{}

The modular branching rules for the symmetric groups $S_n$ over an
algebraically closed field $\mathbb F$ of characteristic $p$ were
obtained by Kleshchev \cite{K1}. Subsequently, the branching graph
of Kleshchev was interpreted by Lascoux, Leclerc, and Thibon as
the crystal graph of the basic representation of the quantum
affine algebra $U_q(\widehat{\mathfrak{sl}}_p)$. Further
connections \cite{LLT} between (affine or cyclotomic) Hecke
algebras of type $A$ at the $\ell$th roots of unity and
Kashiwara-Lusztig crystal basis for integrable
$U_q(\widehat{\mathfrak{sl}}_\ell)$-modules have been
systematically developed by Ariki and Grojnowski from viewpoints
complementary (if perhaps not complimentary) to each other
\cite{Ar1, Gro} (also cf. \cite{Br1, GV, OV}). A parallel version
for degenerate affine Hecke algebra $\mathcal{H}_n$ (introduced by
Drinfeld \cite{Dr} and Lusztig \cite{Lus}) is formulated in
Kleshchev's book \cite{K2}, where the symmetric group algebra
appears as a minimal cyclotomic quotient algebra of
$\mathcal{H}_n$. We refer to \cite{K2} for more references and
historical remarks.

\subsection{}

The goal of this paper is to introduce the wreath Hecke algebra
$\mathcal{H}_n(G)$ associated to an arbitrary finite group $G$,
and to develop its representation theory over the field $\mathbb
F$ of characteristic $p\ge 0$. When $G$ is trivial,
$\mathcal{H}_n(G)$ specializes to the degenerate affine Hecke
algebra $\mathcal{H}_n$. The main results of this paper include
the classification of the simple $\mathcal{H}_n(G)$-modules, the
modular branching rule for $\mathcal{H}_n(G)$, and its
interpretation via crystal graphs of quantum affine algebras.

The modular representations of the spin analogue of the wreath Hecke
algebra in the sense of spin symmetric group (cf. \cite[Part
II]{K2}) will be treated in a separate paper.

\subsection{}

The representation theory of wreath products $G_n=G^n \rtimes S_n$
over $\C$ is known to be largely controlled by infinite-dimensional
Lie algebras \cite{W1, W2, Ze}. Our introduction of the wreath Hecke
algebra $\mathcal{H}_n(G)$ in Section~\ref{wreath Hecke algebra} was
motivated by the desire to study the modular representations of
wreath products over $\mathbb F$. In the wreath Hecke algebra, the
role of $S_n$ is played by the wreath product $G_n$. Moreover, there
exists a canonical surjective algebra homomorphism from
$\mathcal{H}_n(G)$ to the group algebra $\mathbb F G_n$, where the
polynomial generators are mapped to the generalized Jucys-Murphy
elements in $\mathbb F G_n$ (introduced independently in \cite{Pu}
and \cite{W2} with different applications).

The wreath Hecke algebra also arises naturally (cf. \cite{Wan}) in
the centralizer construction of wreath products in the sense of
Molev-Olshanski. For a cyclic group $G=C_r$, the algebra
$\mathcal{H}_n(C_r)$ appeared in Ram and Shepler \cite{RS} in their
search of degenerate (=graded) Hecke algebras associated to complex
reflection groups. Our results on representation theory are new in
this case even when the characteristic of $\mathbb F$ is zero.

We establish the PBW basis of $\mathcal{H}_n(G)$ and identify the
center for $\mathcal{H}_n(G)$ in Section ~\ref{wreath Hecke
algebra}.

\subsection{}

Our study of representation theory of $\mathcal{H}_n(G)$ is built
on an equivalence between the $\Hng$-module category (or rather a
certain full subcategory if $p$ divides the order $|G|$ of the
group $G$) and the module category of an algebra which is a direct
sum of certain products of degenerate affine Hecke algebras of
smaller ranks. This is achieved in Section~\ref{morita equivalence
and simples}.

As a first application of the above category equivalence, the
classification of finite dimensional simple
$\mathcal{H}_n(G)$-modules is obtained in Section~\ref{branching
rule} by a reduction to the known classification of simple modules
for $\mathcal{H}_n$ (cf. \cite{K2}). As a second application, we
establish the modular branching rule for $\mathcal{H}_n(G)$ \`a la
Kleshchev. That is, we describe explicitly the socle of the
restriction of a simple $\mathcal{H}_n(G)$-module to a subalgebra
$\mathcal{H}_{n-1,1}(G)$, and hence to the subalgebra
$\mathcal{H}_{n-1}(G)$ (see Section~\ref{branching rule} for
precise statement and notation).

In Section \ref{crystal operators}, we formulate the cyclotomic
wreath Hecke algebras $\mathcal{H}^{\lambda}_n(G)$ as a family of
finite dimensional quotient algebras of $\mathcal{H}_n(G)$. In
particular, the wreath product group algebra $\mathbb{F}G_n$ appears
as the minimal cyclotomic wreath Hecke algebra. Just as in the
degenerate affine Hecke algebra case, much of the representation
theory of $\mathcal{H}_n(G)$ (e.g. the classification of simple
modules and branching rules) is reduced to that for the cyclotomic
wreath Hecke algebras. We establish an equivalence between (a
distinguished full subcategory of) the module category of a
cyclotomic wreath Hecke algebra and the module category of a certain
variant of the degenerate cyclotomic Hecke algebras.

Now let us assume that $p$ is prime to $|G|$. The ratio of $|G|$
by the degree of a simple $G$-module (which is known to be an
integer) modulo $p$ has come to play a significant role. The
classification of blocks for a cyclotomic wreath Hecke algebra is
reduced to its degenerate cyclotomic counterpart in Brundan
\cite{Br} (a $q$-analogue was due to Lyle and Mathas). For $p>0$,
we define an action of an affine Lie algebra $\mathfrak{g}$ (which
is a direct sum of copies of $\widehat{\mathfrak{sl}}_p$) on the
direct sum of the Grothendieck groups of
$\mathcal{H}_n^{\lambda}(G)$-modules for all $n$ and further show
that the resulting representation is irreducible and integrable.
The modular branching rules for $\mathcal{H}_n^{\lambda}(G)$ are
now controlled by the crystal graph of the integrable
representation of the corresponding quantum affine algebra
$U_q(\mathfrak{g})$. A similar but somewhat more cumbersome
description is available also for $p=0$.

\subsection{Acknowledgments}

This work will form a part of the dissertation of the first author
at University of Virginia. The second author thanks Zongzhu Lin for
a stimulating discussion in 2003 on modular representations of
wreath products.

The second author communicated to S.~Ariki during his visit at
Virginia in 2004 that the modular representation theory of the
degenerate affine Hecke algebra could be generalized to the wreath
product setup. We thank Aaron Phillips (who briefly participated
in the project at an early stage) and Ariki for their interests.
When keeping Ariki updated in March 2008 on the completion of our
project, we learned that in the meantime his student Tsuchioka has
independently worked out the modular branching rules for wreath
products in some recent paper (see our Proposition~\ref{rule:Gn}
and Remark~\ref{rem:wreathcrystal}), and Tsuchioka has also been
obtaining results on a variant of wreath Hecke algebras which
overlap significantly with our paper. Unaware of the reference
\cite{Pu} which Tsuchioka pointed out to us, the first author has
worked out a wreath product generalization of \cite{OV} using
\cite{W2}.

\section{Definition and properties of the wreath Hecke algebra}\label{wreath Hecke algebra}

\subsection{The $p$-regular conjugacy classes of wreath products.}

Let $G$ be a finite group, and let $G_*$ denote the set of all
conjugacy classes of $G$. The symmetric group $S_n$ acts on the
product group $G^n=G\times \cdots\times G$ by permutations:
${}^{w}g:=w (g_1,\ldots, g_n)=(g_{w^{-1}(1)},\ldots,
g_{w^{-1}(n)})$ for any $g=(g_1,\ldots,g_n)\in G^n$ and $w\in
S_n$. The wreath product of $G$ with $S_n$ is defined to be the
semidirect product
$$
G_n=G^n\rtimes S_n=\{(g,w)|g=(g_1,\ldots, g_n)\in G^n, w\in S_n\}
$$
with the multiplication $(g,w)(h,\tau)=(g \cdot {}^{w}h, w\tau)$.

Let $\lambda=(\lambda_1,\ldots, \lambda_l)$ be a partition of
integer $|\lambda|=\lambda_1+\cdots+\lambda_l$, where $\lambda_1\geq
\dots \geq \lambda_l \geq 1$.
We will also write a partition as $
\lambda=(1^{m_1}2^{m_2}\cdots), $ where $m_i$ is the number of parts
in $\lambda$ equal to $i$.

We will use partitions indexed by $G_*$. For a finite set $X$ and
$\rho=(\rho(x))_{x\in X}$ a family of partitions indexed by $X$, we
write
$$\|\rho\|=\sum_{x\in X}|\rho(x)|.$$
Sometimes it is convenient to regard $\rho=(\rho(x))_{x\in X}$ as a
partition-valued function on $X$. We denote by $\mathcal P(X)$ the
set of all partitions indexed by $X$ and by $\mathcal P_n(X)$ the
set of all partitions in $\mathcal P(X)$ such that $\|\rho\|=n$.

The conjugacy classes of $G_n$ can be described as follows. Let
$x=(g, \sigma )\in G_n$, where $g=(g_1, \ldots, g_n) \in G^n,$ $
\sigma \in S_n$. The permutation $\sigma $ is written as a product
of disjoint cycles. For each such cycle $y=(i_1 i_2 \cdots i_k)$
the element $g_{i_k} g_{i_{k -1}} \cdots g_{i_1} \in G$ is
determined up to conjugacy in $G$ by $g$ and $y$, and will be
called the {\em cycle-product} of $g$ corresponding to the cycle
$y$. For any conjugacy class $C$ and each integer $i\geq 1$, the
number of $i$-cycles in $\sigma$ such that the corresponding
cycle-product of $g$ lies in $C$ will be denoted by $m_i(C)$.
Denote by $\rho (C)$ the partition $(1^{m_1 (C)} 2^{m_2 (C)}
\ldots )$, $C \in G_*$. Then each element $x=(g, \sigma)\in G_n$
gives rise to a partition-valued function $( \rho (C))_{C \in G_*}
\in {\mathcal P} ( G_*)$ such that $\sum_{i, C} i m_i(C) =n$. The
partition-valued function $\rho =( \rho(C))_{ C \in G_*} $ is
called the {\em type} of $x$. It is known (cf. \cite{Mac1}) that
any two elements of $G_n$ are conjugate in $G_n$ if and only if
they have the same type.

Denote by $G_{p*}$ the set of conjugacy classes of $G$ whose
elements have order prime to $p$.

\begin{prop}
There is a natural bijection between the set $(G_n)_{p*}$ and the
set
$$
\{ \rho =(\rho(C))_{C \in G_{p*}} \mid
\|\rho\|=n, \rho(C) \text{ has no part divisible by } p \}.$$
\end{prop}
\begin{proof}
For a given $a =(g,\sigma) \in G_n$ with $g=(g_1, \ldots, g_n)\in
G^n$ and $\sigma \in S_n$, clearly the order $o(\sigma)$ of
$\sigma$ divides the order $ o(a)$ of $a$. Set $d =o(\sigma)$. Let
us assume $p \nmid d$. Then
\begin{eqnarray} \label{eq:power}
a^d =(g \cdot {}^{\sigma}g  \cdots  {}^{\sigma^{d-1}}g, 1).
\end{eqnarray}

Fix an index $i$ with $1 \le i \le n$. Let us take a cycle of
$\sigma$, say $y=({i_1} i_2 \ldots {i_k})$, where $i_k=i$. One has
$k \mid d$. Then $y (g_{i_j}) =g_{i_{j-1}}$ for $ 1<j \le k$. So
the $i_k$-th factor of the $n$-tuple $(g \cdot {}^{\sigma}g \cdots
{}^{\sigma^{d-1}}g)$ equals $(g_{i_k}g_{i_{k-1}} \cdots
g_{i_1})^{d/k}$. Note $\gcd (d/k,p)=1$. Thus, the order of $a^d$,
which equals the order of $(g \cdot {}^{\sigma}g  \cdots
{}^{\sigma^{d-1}}g)$ by (\ref{eq:power}), is prime to $p$ if and
only if the cycle product $g_{i_k}g_{i_{k-1}} \cdots g_{i_1}$ lies
in $G_{p*}$ for every cycle $y$.
\end{proof}

\begin{cor}\label{conjugate class equation}
Let $q$ be an indeterminate. We have
$$ \sum_{n=0}^\infty | (G_n)_{p*} | q^n
 = \prod^\infty_{\stackrel{m =1}{p \nmid m}} \left(\frac{1}{1-q^m}\right)^{| G_{p*}
 |}.
$$
\end{cor}

\subsection {Definition of wreath Hecke algebras}

 Let $\mathbb{F}$ be an
algebraically closed field of characteristic $p$ and let
$\mathbb{F}G_n$ be the group algebra of the wreath product $G_n$.
For each $g\in G$ and $1\leq i\leq n$, let $g^{(i)}\in G^n$
correspond to $g$ in the $i$-th factor subgroup of $G^n$. Recall
that the (generalized) {\it Jucys-Murphy elements} $\xi_k\in
\mathbb{F}G_n (1\leq k\leq n)$ are introduced independently in
\cite{Pu} and \cite{W2} as follows:
\begin{eqnarray*}
\xi_k: =\sum_{1\leq i<k}\sum_{g\in G} \left(g^{(i)}(g^{-1})^{(k)},
(i,k) \right).
\end{eqnarray*}
If $G=\{1\}$, then $G_n =S_n$, and the $\xi_k$ become the usual
Jucys-Murphy elements \cite{Ju, Mu}.

Recall that $S_n$ is generated by the simple reflections
$s_1,\ldots,s_{n-1}$. Denote
\begin{eqnarray} \label{t_ij}
t_{ij}=\sum_{h\in G}h^{(i)}(h^{-1})^{(j)}\in\mathbb{F}G^n, \quad
1\leq i<j\leq n.
\end{eqnarray}
The following proposition follows by a direct computation.
\begin{prop}\label{Jucys murphy}
The following identities hold in the group algebra $\mathbb{F}G_n$:
\begin{align}
 \xi_i\xi_j&=\xi_j \xi_i,\quad 1\leq i,j \leq n,\notag \\
 g\xi_i&=\xi_i g,\quad g\in G^n, 1\leq i\leq n,\notag\\
 s_i\xi_i&=\xi_{i+1}s_i-t_{i,i+1}, \quad 1\leq i\leq n-1,\notag\\
 s_i\xi_j&=\xi_js_i,\quad j\neq i,i+1.\notag
\end{align}
\end{prop}

Let $P_n=\mathbb{F}[x_1,\ldots,x_n]$ be the algebra of polynomials
in $x_1,\ldots,x_n$. For each
$\alpha=(\alpha_1,\ldots,\alpha_n)\,\in\mathbb{Z}_+^n$, set
$x^{\alpha}=x_1^{\alpha_1}\cdots x_n^{\alpha_n}.$ The symmetric
group $S_n$ acts as automorphisms on $P_n$ by permutation. Let us
denote this action by $f\mapsto {}^wf$ for $w\in S_n$ and $f\in
P_n$. Then we have ${}^w(x^{\alpha})=x^{w\alpha}$, where
$w\alpha=(\alpha_{w^{-1}1},\ldots,\alpha_{w^{-1}n})$ for
$\alpha=(\alpha_1,\ldots,\alpha_n)\in\mathbb{Z}_+^n$ and $w\in
S_n$.

\begin{defn}
The {\em wreath Hecke algebra }$\mathcal{H}_n(G)$ is an associative
algebra over $\mathbb{F}$ generated by $G^n, s_1,\ldots,s_{n-1}$ and
$x_1,\ldots,x_n$ subject to the following relations:
\begin{align}
 x_ix_j&=x_jx_i,\quad 1\leq i,j \leq n,\notag\\
x_ig&=gx_i, \quad g\in G^n, 1\leq i\leq n, \label{comm:xg} \\
s_i^2=1,\quad s_is_j &=s_js_i, \quad
s_is_{i+1}s_i=s_{i+1}s_is_{i+1}, \quad|i-j|>1,\label{braid}\\
s_ix_i&=x_{i+1}s_i-t_{i,i+1},\label{symmetric and polnomial 2}\\
s_ix_j&=x_js_i, \quad j\neq i, i+1, \label{Sn x}\\
s_ig&={}^{s_i}g\, s_i, \quad g\in G^n, 1\leq i\leq
n-1.\label{SnGn}
\end{align}
\end{defn}

\begin{rem}\label{surj. hom. to Gn}
By Proposition \ref{Jucys murphy}, we have a surjective algebra
homomorphism from $\mathcal{H}_n(G)$ to $\mathbb{F}G_n$, which is an
extension of the identity map of $\mathbb{F}G_n$ and sends each
$x_k$ to $\xi_k$ for $1\leq k\leq n$. This was our original
motivation for the definition of $\mathcal{H}_n(G).$ For $n=2$, the
algebra $\mathcal{H}_2(G)$ can also be found in \cite[Section
3]{Pu}.

For a cyclic group $G=C_r$, the algebra $\mathcal{H}_n(C_r)$ also
appeared in \cite{RS}. Moreover, it is observed \cite{De} that
$\mathcal{H}_n(C_r)$ appears naturally as a subalgebra of the
symplectic reflection algebra of Etingof-Ginzburg associated to
the complex reflection group $G(r,1,n)$. It will be interesting to
see if $\mathcal{H}_n(G)$ associated to a general finite subgroup
$G$ of $SL_2(\mathbb{C})$ is related to symplectic reflection
algebras.
\end{rem}

If $G=\{1\}$ is the trivial group, then
$\mathcal{H}_n(G)=\mathcal{H}_n$, the degenerate affine Hecke
algebra for $S_n$ \cite[Chapter 3]{K1}, where \eqref{symmetric and
polnomial 2} is replaced by the relation
\begin{equation}
s_ix_i =x_{i+1}s_i-u
\end{equation}
with $u=1$. We shall denote by $\Hno$ the algebra $P_n \rtimes \F
S_n$, i.e., the degenerate affine Hecke algebra with $u=0$.

\begin{lem}
For $f \in P_n$, $g \in G^n$, and $1\leq i\leq n-1$, the following
identities hold in $\mathcal{H}_n(G)$:
 \begin{align}
s_it_{i,i+1}&=t_{i,i+1}s_i \label{s and t},\\
s_ix_{i+1}&=x_is_i+t_{i,i+1}\label{symmetric and polynomial 3},\\
s_igf&= {}^{s_i}g
\left({}^{s_i}fs_i+t_{i,i+1}\frac{f-{}^{s_i}f}{x_{i+1}-x_i}
   \right) \label{symmetric group and polynomial},\\
t_{i,i+1}g&={}^{s_i}g\, t_{i,i+1}. \label{t and g}
\end{align}
\end{lem}
\begin{proof}
The equation (\ref{s and t}) follows from the identity
$s_ig^{(i)}(g^{-1})^{(i+1)}=(g^{-1})^{(i)}g^{(i+1)}s_i$, while
(\ref{symmetric and polynomial 3}) follows from (\ref{braid}),
(\ref{symmetric and polnomial 2}) and (\ref{s and t}). The equation
(\ref{symmetric group and polynomial}) is deduced by induction on
the degree of the polynomial $f$.

The equation (\ref{t and g}) is reduced to the case $n=2$ and $i=1$,
and hence a computation in $\mathbb{F}G^2$. Indeed, $
t_{12}(h_1,h_2)=\sum_{h\in G}(hh_1,h^{-1}h_2)=\sum_{g\in
G}(h_2g,h_1g^{-1}) =(h_2,h_1)t_{12}$ for any $h_1, h_2\in G$, where
we have used a substitution $g =h_2^{-1}hh_1$.
\end{proof}

\subsection{The PBW basis for $\mathcal{H}_n(G)$}

The following lemma follows from (\ref{symmetric group and
polynomial}).

\begin{lem}\label{spanned set}
Let $x^{\alpha}\in P_n, w\in S_n, g=(g_1,\ldots,g_n)\in G^n$,
where $\alpha=(\alpha_1,\ldots,\alpha_n)\in\mathbb{Z}^n_+$, and
denote the Bruhat ordering on $S_n$ by $\leq$. Then in
$\mathcal{H}_n(G)$ we have
$$wgx^{\alpha} =({}^wg)x^{w\alpha}w+\sum_{u<w}g_uf_uu,
\qquad
 gx^{\alpha}w =w({}^{w^{-1}}g)x^{w^{-1}\alpha}
+\sum_{u<w}ug_u^{\prime}f_u^{\prime}$$ for some $f_u,
f_u^{\prime}\in P_n$ of degrees less than the degree of $x^{\alpha}$
and $g_u, g_u^{\prime}\in \mathbb{F}G^n$.
\end{lem}

\begin{thm}\label{PBW basis}
The multiplication of algebras induces an isomorphism of vector
spaces:
$$
P_n \otimes \mathbb{F}G^n \otimes \mathbb{F}S_n \longrightarrow
\mathcal{H}_n(G).
$$
That is, the elements $\{x^{\alpha}gw|~ \alpha\in\mathbb{Z}_+^n,
g\in G^n, w\in S_n\}$ form a linear basis for $\mathcal{H}_n(G)$
(which is called the PBW basis).
\end{thm}

\begin{proof}
It follows easily from Lemma \ref{spanned set} that
$\mathcal{H}_n(G)$ is spanned by the elements $x^{\alpha}gw$ for
$\alpha\in\mathbb{Z}_+^n, g\in G^n, w\in S_n$. Note that
$\{h\otimes y^{\alpha}|~ h\in G^n, \alpha\in\mathbb{Z}_+^n\}$
forms a basis for the vector space
$\mathbb{F}G^n\otimes_{\mathbb{F}}\mathbb{F}[y_1,y_2,\ldots,y_n]$.
We can verify by a direct yet lengthy computation that
$\mathbb{F}G^n\otimes_{\mathbb{F}}\mathbb{F}[y_1,y_2,\ldots,y_n]$
is an $\mathcal{H}_n(G)$-module via
\begin{align}x_i\circ(h\otimes
y^{\alpha})&=h\otimes y_iy^{\alpha}, \quad 1\leq i\leq
n,\notag\\
g\circ (h\otimes y^{\alpha})&=gh\otimes y^{\alpha}, \quad g\in
G^n,\notag\\
s_j\circ (h\otimes y^{\alpha}) &={}^{s_j}h\otimes y^{s_j\alpha}+
({}^{s_j}h)t_{j,j+1} \otimes\frac{y^{\alpha}-
y^{s_j\alpha}}{y_{j+1}-y_j}, \quad 1\le j \le n-1. \notag
\end{align}
In the process of verification, the following identities in
$\mathcal{H}_n(G)$ for $1\leq i\leq n-2$ are used:
\begin{align}
t_{i,i+2}t_{i+1,i+2}&=t_{i,i+1}t_{i,i+2}=t_{i+1,i+2}t_{i,i+1},\notag\\
t_{i,i+1}t_{i+1,i+2}&=t_{i,i+2}t_{i,i+1}=t_{i+1,i+2}t_{i,i+2},\notag\\
t_{i+1,i+2}t_{i,i+2} t_{i,i+1}
&=t_{i,i+1}t_{i,i+2}t_{i+1,i+2}.\notag
\end{align}
To see that the elements $x^{\alpha}gw$ are linearly independent, it
suffices to show that they act by linearly independent linear
operators on
$\mathbb{F}G^n\otimes_{\mathbb{F}}\mathbb{F}[y_1,y_2,\ldots,y_n]$.
This is clear if we consider the action on an element of the form
$y_1^Ny_2^{2N}\cdots y_n^{nN}$ for $N\gg0$.
\end{proof}

By Theorem \ref{PBW basis}, we can from now on identify $P_n$,
$\mathbb{F}G^n, \mathbb{F}S_n$ and $\mathbb{F}G_n$ with the
corresponding subalgebras of $\mathcal{H}_n(G)$. Let $P_n(G)$ be the
subalgebra generated by $G^n$ and $x_1,\ldots,x_n$, then
$$P_n(G)\cong \mathbb{F}G^n\otimes P_n.$$ Also, if $m\leq n$, we
regard $\mathcal{H}_m(G)$ as the subalgebra of $\mathcal{H}_n(G)$
generated by $G^m, x_1,\ldots,x_m$ and $s_1,\ldots,s_{m-1}.$
\subsection{The center of $\mathcal{H}_n(G)$}

We start with a preparatory lemma.
\begin{lem}\label{center 0}
The center of $\mathcal{H}_n(G)$ is contained in the subalgebra
$P_n(G)$.
\end{lem}

\begin{proof}
Take a central element $z=\sum_{w\in S_n}z_ww\in\mathcal{H}_n(G),$
where $z_w=\sum d_{g,\alpha}gx^{\alpha}\in P_n(G)$. Let $\tau$ be
maximal with respect to the Bruhat order such that $z_{\tau}\neq 0$.
Assume $\tau\neq 1$. Then there exists $i\in\{1,2,\ldots,n\}$ with
$\tau(i)\neq i$. Then by Lemma \ref{spanned set},
$$
x_iz-zx_i=z_{\tau}(x_i-x_{\tau i})\tau+\sum a_{g,
\alpha,w}gx^{\alpha}w,
$$
where the sum is over $g\in G^n, x^{\alpha}\in P_n$ and $w\in S_n$
with $w\ngeq \tau$ in the Bruhat order. So, by Theorem \ref{PBW
basis}, $z_{\tau}=0$ which is a contradiction. Hence, we must have
$\tau=1$ and $z\in P_n(G)$.
\end{proof}

Let $G_*=\{C_1,\ldots, C_s\}$ denote the set of all conjugacy
classes of $G$. We set
\begin{eqnarray*} \label{eq:I}
\mathcal{I}=\{\underline{i}=(i_1,\ldots,i_n) \mid 1\leq
i_1,\ldots,i_n \leq s \}
\end{eqnarray*}
with an $S_n$-action given by $\sigma
\underline{i}=(i_{\sigma^{-1}1},\ldots,i_{\sigma^{-1}n})$ for
$\sigma \in S_n$. Then the set $(G^n)_*$ of conjugacy classes of
$G^n$ is
$$
C_{\underline{i}} :=\{g=(g_1,\ldots,g_n) \mid g_k\in C_{i_k},1\leq
k\leq n\}, \quad \underline{i}\in\mathcal{I}.
$$
We shall denote the class sum
$$
\overline{C_{\underline{i}}} : =\sum_{g\in C_{\underline{i}}}g \in
\mathbb F G^n \subset \Hng.
$$

By Lemma \ref{center 0}, a central element $z$ of $\mathcal{H}_n(G)$
is of the form $z=\sum_{g\in G^n,\al
\in\mathbb{Z}^n_+}d_{g,\alpha}gx^{\alpha},$ where
$d_{g,\alpha}\in\mathbb{F}$. It follows from
$hx^{\alpha}=x^{\alpha}h$ and $hz=zh$ that
$d_{g,\alpha}=d_{hgh^{-1},\alpha}$ for all $g, h\in G^n$. Hence, $z$
can be written as
\begin{eqnarray} \label{eq:central}
z =\sum_{\underline{i} \in \mathcal I,\alpha \in\mathbb{Z}^n_+}
d_{\underline{i},\alpha} x^{\alpha}\overline{C_{\underline{i}}},
\qquad d_{\underline{i},\alpha} \in \mathbb F.
\end{eqnarray}

\begin{thm}\label{center}
The center of $\mathcal{H}_n(G)$ consists of elements of the form
\eqref{eq:central} whose coefficients $d_{\underline{i},\alpha}$ are
$S_n$-invariant, i.e., $ d_{w \underline{i}, w \alpha}
=d_{\underline{i},\alpha}$ for all $w\in S_n, \underline{i} \in
\mathcal I,$ and $\alpha \in\mathbb{Z}^n_+.$
\end{thm}
\begin{proof}
Take a central element $z \in \Hng$ of the form \eqref{eq:central}.
Applying (\ref{symmetric group and polynomial}), we get
$$
s_1 z
=\sum_{\underline{i},\alpha}d_{\underline{i},\alpha}(x^{s_1\alpha})
({}^{s_1}\overline{C_{\underline{i}}})s_1
+\sum_{\underline{i},\alpha}d_{\underline{i},\alpha} t_{12}
\frac{x^{\alpha}-  x^{s_1\alpha} }{x_{2}-x_1}
\overline{C_{\underline{i}}} \,.
$$
By Theorem~\ref{PBW basis}, $s_1z=zs_1$ (or rather $s_1zs_1 =z$) is
equivalent to identities \eqref{d1}-\eqref{d2}:
\begin{align}
\sum_{\underline{i},\alpha}d_{\underline{i},\alpha} (
x^{s_1\alpha}) ({}^{s_1}\overline{C_{\underline{i}}})
&=\sum_{\underline{i},\alpha} d_{\underline{i},\alpha}x^{\alpha}
\overline{C_{\underline{i}}} \;\; ( =: z), \label{d1}
 \\
\sum_{\underline{i},\alpha}d_{\underline{i},\alpha}t_{12}\frac{x^{\alpha}-
 x^{s_1\alpha} }{x_2-x_1}\overline{C_{\underline{i}}}&=0.
 \label{d2}
\end{align}

We claim that (\ref{d1}) implies (\ref{d2}). Indeed, assuming
(\ref{d1}) we obtain that
\begin{align*}
t_{12} z & = zt_{12} \quad \text{ since } z \text{ is central}, \\
&=\sum_{\underline{i},\alpha}d_{\underline{i},\alpha} (
x^{s_1\alpha}) ({}^{s_1}\overline{C_{\underline{i}}} )
t_{12} \quad \text{ by } \eqref{d1}, \\
&=t_{12} \sum_{\underline{i},\alpha}d_{\underline{i},\alpha} (
x^{s_1\alpha}) \overline{C_{\underline{i}}} \quad \text{ by }
\eqref{comm:xg} \text{ and } (\ref{t and g}).
\end{align*}
This is a variant of \eqref{d2} with the denominator cleared.

Note now that (\ref{d1}) holds if and only if $d_{s_1
\underline{i},s_1 \alpha} =d_{\underline{i},\alpha}$. Applying the
same procedure to $s_kz=zs_k$, we obtain that $d_{s_k
\underline{i},s_k \alpha} =d_{\underline{i},\alpha}$ and hence $d_{w
\underline{i},w \alpha} =d_{\underline{i},\alpha}$ for any $w\in
S_n$.

Reversing the above arguments, an element $z \in \Hng$ of the form
\eqref{eq:central} satisfying the $S_n$-invariant property $d_{w
\underline{i},w \alpha} =d_{\underline{i},\alpha}$ is indeed
central.
\end{proof}

\begin{rem}
By Theorem~\ref{center}, the center of $\mathcal{H}_n(G)$ contains
the ring $\Lambda_n$ of symmetric polynomials in $x_1,\ldots,
x_n$. Hence by Theorem~\ref{PBW basis} the algebra $\Hng$ is
finitely generated as a module over its center, which implies that
every simple $\Hng$-module is finite-dimensional.
\end{rem}
\section{An equivalence of module categories}\label{morita equivalence and simples}

In this section, we establish a key category equivalence which
relates the wreath Hecke algebra to degenerate affine Hecke
algebras.

\subsection{A useful lemma}

Let $G^*=\{V_1,\ldots,V_r\}$ be a complete set of pairwise
non-isomorphic finite dimensional simple $\mathbb{F}G$-modules, and
set
$$
\text{dim}_{\mathbb{F}}V_k=d_k, \qquad 1\leq k\leq r.
$$
It is know by elementary Clifford theory that the
$\mathbb{F}G^2$-module $V_k^{\otimes 2}$ affords a simple
$\mathbb{F}G_2$-module by letting $s_1 =(1 \, 2)$ act as the
operator $P$ which permutes the two tensor factors. Also, for
$1\leq k\neq l\leq r$,
$\text{ind}^{\mathbb{F}G_2}_{\mathbb{F}G^2}(V_k\otimes V_l)
=V_k\otimes V_l\oplus V_l\otimes V_k$ is a simple
$\mathbb{F}G_2$-module where $(1\,2)$ acts as the permuting
operator $P$.

\begin{lem}\label{weig-t-0}
Retain the above notations. Then,
\begin{enumerate}
\item $t_{12}=0$ when acting on a simple $\mathbb{F}G^2$-module
$V_k\otimes V_l$ for $1\leq k\neq l\leq r$.
 \item
$t_{12} =c_k P$ when acting on the $\mathbb{F}G^2$-module
$V_k^{\otimes 2}$, where the scalar $c_k\in\mathbb{F}$ satisfies
$d_k c_k= |G|$ in $\F$.
\end{enumerate}
\end{lem}

\begin{proof}
Note that the Jucys-Murphy element $\xi_2=\sum_{g\in
G}((g,g^{-1}),(1\,2))$ is central in $\mathbb{F}G_2$. By Schur's
Lemma, $\xi_2$ acts on the simple $\mathbb{F}G_2$-modules
$V_k^{\otimes 2}$ and
$\text{ind}^{\mathbb{F}G_2}_{\mathbb{F}G^2}(V_k\otimes V_l)$ as
scalars. On $\text{ind}^{\mathbb{F}G_2}_{\mathbb{F}G^2}(V_k\otimes
V_l)$, $\xi_2$ maps the subspace $V_k\otimes V_l$ to the subspace
$V_l\otimes V_k$, and hence $\xi_2$ acts as zero. Then, $t_{12}
=(1\,2)\xi_2$ acts as zero on
$\text{ind}^{\mathbb{F}G_2}_{\mathbb{F}G^2}(V_k\otimes V_l)$, and
hence as zero on $V_k\otimes V_l$. This proves (1).

Assume that $\xi_2$ acts on $V_k^{\otimes 2}$ as a scalar $c_k$.
Then on one hand, the trace of $\xi_2$ on $V_k^{\otimes 2}$ is
$d_k^2c_k$, and on the other hand, it is also equal to $|G|\cdot
\text{Tr}|_{V_k^{\otimes 2}}(1\,2) $ since $((g,g^{-1}),(12))$ is
conjugate to $(12)$ in $G_2$ for each $g\in G$. Note that
$\text{Tr}|_{V_k^{\otimes 2}}(1\,2)=d_k$ since $(12)$ acts as the
permutation operator $P$. Therefore we have
$$
d_k^2c_k=|G|\cdot \text{Tr}|_{V_k^{\otimes 2}}(1\,2) =d_k|G|,
$$
which is equivalent to $d_k c_k= |G|$. Now (2) follows by noting
again $t_{12} =(1\,2)\xi_2$.
\end{proof}

\begin{rem}
If $p$ is prime to $|G|$ then $c_k= |G|/d_k\in\mathbb{I}-\{0\}$.
However  when $p$ divides $|G|$ it is possible that $c_k=0$ (e.g.
for the trivial module). It is also possible that $d_k$ does not
divide $|G|$ (in this case $c_k=0$ too). Let $p=7$ and
$G=SL(2,\mathbb{F}_7)$, which is of order $6\cdot 7\cdot 8$. An
irreducible $\F G$-module $V_m$ of dimension $m$, for each $1\leq
m< 7$, is given by the $\F$-vector space of homogeneous
polynomials of degree $m-1$ in two variables. This example is
kindly provided by L.~Scott.


\end{rem}

Clearly, $\{V_{i_1}\otimes\cdots\otimes V_{i_n}| 1\leq
i_1,\ldots,i_n\leq r\}$ forms a complete set of pairwise
non-isomorphic simple $\mathbb{F}G^n$-modules. Denote by $P_{kl}$
the operator on $V_{i_1}\otimes\cdots\otimes V_{i_n}$ which permutes
the $k$th and $l$th factors. Recall the definition of $t_{kl}$ from
\eqref{t_ij}.

\begin{cor} \label{weig-t-1}
On $V_{i_1}\otimes\cdots\otimes V_{i_n}$, $t_{kl}$ acts as $c_{i_k}
P_{kl}$ if $i_k = i_l$; otherwise $t_{kl}$ acts as zero.
\end{cor}

\subsection{Structure of $\mathcal{H}_n(G)$-modules}
\label{sec:Clifford}

Set
$$
\mathbb{I}:=\mathbb{Z}\cdot1\subset\mathbb{F}.
$$
That is, $\mathbb{I} =\{0, 1, \ldots, p-1\}$ for $p>0$ and
$\mathbb{I} =\Z$ for $p=0$.

For an algebra $R$, we denote by $R$-{\bf mod} the category of
finite dimensional left $R$-modules. Denote by $\HngMods$ the full
subcategory of $\mathcal{H}_n(G)$-{\bf mod} consisting of finite
dimensional $\mathcal{H}_n(G)$-modules which are semisimple when
restricted to the subalgebra $\mathbb{F}G^n$. Denote the set of
$r$-tuple compositions of $n$ by
$$
\mathcal{C}_n^r:=\{\widehat{n}=(n_1,\ldots,n_r) \mid
n_1,\ldots,n_r\in \Z_+, n_1+\cdots+n_r=n\}.
$$
For each $\widehat{n}\in\mathcal{C}_n^r$, let
$V(\widehat{n})=V_1^{\otimes n_1}\otimes\cdots\otimes V_r^{\otimes
n_r}$ be the corresponding simple $\mathbb{F}G^n$-module.
Moreover, denote by $S_{\widehat{n}} = S_{n_1}\times\cdots\times
S_{n_r}$ be the corresponding Young subgroup of $S_n$ and let
$\Theta(\widehat{n})$ be a complete set of representatives of left
cosets of $S_{\widehat{n}}$ in $S_n$.

Define $\mathcal{H}_{\widehat{n}}(G)$ to be the subalgebra of
$\mathcal{H}_n(G)$ generated by $G^n,x_1,\ldots,x_n$ and
$S_{\widehat{n}}$. Then
$$
\mathcal{H}_{\widehat{n}}(G) \cong\mathcal{H}_{n_1}(G)
\otimes\mathcal{H}_{n_2}(G) \otimes\cdots\otimes
\mathcal{H}_{n_r}(G).
$$
For $G=\{1\}$, we drop $G$ and denote $\mathcal{H}_{\widehat{n}}
=\mathcal{H}_{\widehat{n}}(G)$. We denote by $\HnghatMods$ the
full subcategory of $\mathcal{H}_{\n}(G)$-{\bf mod} consisting of
finite dimensional $\mathcal{H}_{\n}(G)$-modules which are
semisimple when restricted to the subalgebra $\mathbb{F}G^n$.

For $M\in \mathcal{H}_n(G)$-{\bf mod}$^s$, let $I_{\widehat{n}}M$ be
the isotypical subspace of $V(\widehat{n})$ in $M$, that is, the sum
of all simple $\mathbb{F}G^n$-submodule of $M$ isomorphic to
$V(\widehat{n})$. Denote
$$
M_{\widehat{n}} :=\sum_{\pi\in S_n} \pi (I_{\widehat{n}}M).
$$

\begin{lem} \label{ealphaM}
Let $\widehat{n}\in\mathcal{C}_n^r$ and $M\in\mathcal{H}_n(G)$-{\bf
mod}$^s$. Then, $I_{\widehat{n}}M$ is an
$\mathcal{H}_{\widehat{n}}(G)$-submodule and $M_{\widehat{n}}$ is an
$\mathcal{H}_n(G)$-submodule of $M$. Moreover, $M_{\widehat{n}}\cong
\text{ind}^{\mathcal{H}_n(G)}_{\mathcal{H}_{\widehat{n}}(G)}
(I_{\widehat{n}}M)$.
\end{lem}
\begin{proof}
Being commutative with $\mathbb{F}G^n$, each $x_i$ maps a simple
$\mathbb{F}G^n$-submodule of $M$ either to zero or to an isomorphic
copy. Hence $I_{\widehat{n}}M$ is invariant under the action of the
subalgebra $P_n$. Since each $\pi\in S_{\widehat{n}}$ maps a simple
$\mathbb{F}G^n$-submodule of $M$ isomorphic to $V(\widehat{n})$ to
another isomorphic one, $I_{\widehat{n}}M$ is invariant under the
action of $S_{\n}$. Hence $I_{\widehat{n}}M$ is an
$\mathcal{H}_{\widehat{n}}(G)$-submodule, since
$\mathcal{H}_{\widehat{n}}(G)$ is generated by $\F G^n, P_n$ and
$S_{\n}.$

It then follows from definition that $M_{\widehat{n}}$ is an
$\mathcal{H}_n(G)$-submodule of $M$.

We have a nonzero $\mathcal{H}_n(G)$-homomorphism $\phi:
\text{ind}^{\mathcal{H}_n(G)}_{\mathcal{H}_{\widehat{n}}(G)}
I_{\widehat{n}}M\rightarrow M_{\widehat{n}}$ by Frobenius
reciprocity. Observe that
$$
M_{\widehat{n}}=\sum_{\pi\in S_n}\pi(I_{\widehat{n}}M)
 =\bigoplus_{\tau\in\Theta(\widehat{n})}\tau(I_{\widehat{n}}M).
$$
Hence $\phi$ is surjective, and then an isomorphism by a dimension
counting argument.
\end{proof}

\begin{lem}\label{Malpha}
We have the following decomposition in $\HngMods$:
$$
M=\bigoplus_{\widehat{n}\in\mathcal{C}_n^r}M_{\widehat{n}}.
$$
\end{lem}

\begin{proof}
Let $M \in \HngMods$. By definition,  as an $\F G^n$-module $M$ is
semisimple. Observe that each $M_{\n}$ is the direct sum of those
isotypical components of simple $\F G^n$-modules which contain
exactly $n_i$ tensor factors isomorphic to $V_i$ for $1\le i \le r$.
Now the lemma follows.
\end{proof}

\subsection{Algebras $A_{\n,r}$ versus $\mathcal{H}_{\widehat{n}}(G)$}

We define the following algebras
\begin{eqnarray} \label{Anr}
A_{n,r} =\bigoplus_{\widehat{n} \in\mathcal{C}_n^r} A_{\n,r},
\qquad A_{\n,r} = {}^{c_1}{\mathcal{H}}_{n_1} \otimes\cdots\otimes
{}^{c_r}{\mathcal{H}}_{n_r},
\end{eqnarray}
where, for $1\le k \le r$, we denote
\begin{eqnarray*}
{}^{c_k}{\mathcal{H}}_{n_k} = \left \{
 \begin{array}{ll}
 {\mathcal{H}}_{n_k}, & \text{ if } c_k \neq 0 \\
 {}^0 {\mathcal{H}}_{n_k}, & \text{ if } c_k =0.
 \end{array}
 \right.
\end{eqnarray*}

Below, we shall denote the polynomial generators in $\Hn$ and in
$A_{n,r}$ by $y_1,\ldots,y_n$ to distinguish from $x_1,\ldots,x_n$
in $\mathcal{H}_n(G)$.

For each $\n\in\mathcal{C}_n^r$, let us denote by $l_k$ the unique
integer such that $n_1+\cdots+n_{l_k-1}+1\leq k\leq
n_1+\cdots+n_{l_k}$, for each $1\le k \le n$.

\begin{prop}\label{Hom functor}
Let $\widehat{n}\in\mathcal{C}_n^r$ and $N \in \HnghatMods$. Then
$\text{Hom}_{\mathbb{F}G^n}(V(\n), N)$ is an
$A_{\widehat{n},r}$-module by letting
$$
\aligned(\pi\diamond\phi)(v_1\otimes\cdots\otimes v_n)
&=\pi\phi(v_{\pi(1)}\otimes\cdots\otimes
v_{\pi(n)}),\\
(y_k\diamond\phi)(v_1\otimes \cdots\otimes v_n) &= \left \{
\begin{array}{ll}
\frac{1}{c_{l_k}}x_k\phi(v_1\otimes\cdots\otimes v_n), & \text{ if
 } c_{l_k}\neq 0 ,\\
x_k\phi(v_1\otimes\cdots\otimes v_n), & \text{ if
 } c_{l_k}=0.
 \end{array}
 \right.
\endaligned
$$
for $\pi\in S_{\widehat{n}}, v_1\otimes\cdots\otimes v_n\in V(\n),
\phi\in\text{Hom}_{\mathbb{F}G^n}(V(\n),N)$ and $1\leq k\leq n.$
Hence, $\text{Hom}_{\mathbb{F}G^n}(V(\n), -)$ is a functor from
$\HnghatMods$ to $A_{\widehat{n},r}$-{\bf mod}.
\end{prop}

\begin{proof}
Let us first show that $\pi\diamond\phi$ is an
$\mathbb{F}G^n$-homomorphism (and skip a similar proof that
$y_k\diamond\phi$ is $\mathbb{F}G^n$-homomorphism). Indeed, for
$g=(g_1,\ldots,g_n)\in G^n$,
$$
\aligned(\pi\diamond\phi) (g(v_1\otimes\cdots\otimes v_n))
 &=\pi\phi(g_{\pi(1)}v_{\pi(1)}\otimes\cdots\otimes
g_{\pi(n)}v_{\pi(n)})\\
&=\pi((g_{\pi(1)},\ldots,g_{\pi(n)})\phi(v_{\pi(1)}
\otimes\cdots\otimes v_{\pi(n)}))\\
&=(g_1,\ldots,g_n)\pi\phi(v_{\pi(1)}\otimes\cdots\otimes
v_{\pi(n)}))\\&=g((\pi\diamond\phi)(v_1\otimes\cdots\otimes v_n)).
\endaligned
$$

Take $1\le k \le r$ such that $k\neq n_1,
n_1+n_2,\ldots,n_1+n_2+\cdots +n_{r-1}$. By Corollary
\ref{weig-t-1},
$$
t_{k,k+1}(v_1\otimes\cdots\otimes
v_n)=c_{l_k}(v_1\otimes\cdots\otimes v_{k-1}\otimes v_{k+1}\otimes
v_{k}\otimes\cdots\otimes v_n).
$$
By definition,
\begin{align*}
((y_ks_k)\diamond \phi)(v_1\otimes\cdots\otimes v_n) &= \left \{
\begin{array}{ll}
\frac{x_k}{c_{l_k}}s_k\phi(v_1\otimes\cdots\otimes v_{k+1}\otimes
v_k\otimes\cdots v_n), & \text{ if }c_{l_k}\neq 0
 \\
 x_ks_k\phi(v_1\otimes\cdots\otimes v_{k+1}\otimes
v_k\otimes\cdots v_n), & \text{ if }c_{l_k}= 0,
\end{array} \right.
 \\
(s_ky_{k+1})\diamond \phi)(v_1\otimes\cdots\otimes v_n) &=\left \{
\begin{array}{ll}
s_k\frac{x_{k+1}}{c_{l_{k+1}}}\phi(v_1\otimes\cdots\otimes
v_{k+1}\otimes v_k\otimes\cdots v_n), &\text{ if }c_{l_k}\neq 0
 \\
s_kx_{k+1}\phi(v_1\otimes\cdots\otimes v_{k+1}\otimes
v_k\otimes\cdots v_n), &\text{ if }c_{l_k}=0.
\end{array} \right.
\end{align*}
The above computations together with \eqref{symmetric and polnomial
2} now imply that $y_ks_k=s_ky_{k+1}-1$ if $c_{l_k}\neq 0$ and
$y_ks_k=s_ky_{k+1}$ if $c_{l_k}=0$.

The other relations for the $A_{\widehat{n},r}$-module structure on
$\text{Hom}_{\mathbb{F}G^n}(V(\n), N)$ are clear.
\end{proof}

\begin{prop}\label{group tensor affine 3}
Let $N$ be an $A_{\widehat{n},r}$-module. Then
$V(\widehat{n})\otimes N$ is an
$\mathcal{H}_{\widehat{n}}(G)$-module via
\begin{align}
g\ast(v_1\otimes v_2\cdots\otimes v_n \otimes z)
 &=g(v_1\otimes v_2\cdots\otimes v_n)\otimes z\notag,
 \\
\pi\ast(v_1\otimes v_2\cdots\otimes v_n\otimes z)
&=v_{\pi^{-1}{1}}\otimes v_{\pi^{-1}2}\otimes\cdots\otimes
v_{\pi^{-1}r}\otimes \pi z \notag,
\\
x_k\ast(v_1\otimes v_2 \cdots \otimes v_n\otimes z) &=\left \{
\begin{array}{ll}
c_{l_k}v_1\otimes v_2\cdots\otimes v_n\otimes y_k z, & \text{ if }
c_{l_k}\neq 0 \notag
\\
v_1\otimes v_2\cdots\otimes v_n\otimes y_k z, & \text{ if } c_{l_k}=
0 \notag,
\end{array} \right.
\end{align}
for $g\in G^n, \pi\in S_{\widehat{n}}, 1\leq k\leq n,
v_1\otimes\cdots\otimes v_n\in V(\widehat{n})$, and $z\in N$.

There exists an isomorphism of $A_{\widehat{n},r}$-modules $\Phi:
N\rightarrow\text{Hom}_{\mathbb{F}G^n}(V(\widehat{n}),
V(\widehat{n})\otimes N)$ given by $\Phi(z) (v) = v\otimes z$.
Moreover, $V(\widehat{n})\otimes N$ is a simple
$\mathcal{H}_{\widehat{n}}(G)$-module if and only if $N$ is a simple
$A_{\widehat{n},r}$-module.
\end{prop}
\begin{proof}
It is straightforward to verify that $V(\widehat{n})\otimes N$ is an
$\mathcal{H}_{\widehat{n}}(G)$-module as given above.

Clearly $\Phi$ is a well-defined injective
$A_{\widehat{n},r}$-homomorphism. On the other hand, observe that as
an $\mathbb{F}G^n$-module, $V(\widehat{n}) \otimes N$ is isomorphic
to a direct sum of copies of $V(\widehat{n})$. Thus $\Phi$ is an
isomorphism by a dimension comparison.

Suppose that $V(\widehat{n})\otimes N$ is a simple
$\mathcal{H}_{\widehat{n}}(G)$-module and $E$ is a
$A_{\widehat{n},r}$-submodule of $N$. Then $V(\widehat{n})\otimes E$
is a $\mathcal{H}_{\widehat{n}}(G)$-submodule of
$V(\widehat{n})\otimes N$, which implies $E=N$. Conversely, suppose
that $N$ is a simple $A_{\widehat{n},r}$-module and $M$ is a nonzero
$\mathcal{H}_{\widehat{n}}(G)$-submodule of $V(\widehat{n})\otimes
N$. Then by Proposition \ref{Hom functor},
$\text{Hom}_{\mathbb{F}G^n}(V(\widehat{n}),M)$ is a nonzero
$A_{\widehat{n},r}$-submodule of
$\text{Hom}_{\mathbb{F}G^n}(V(\widehat{n}), V(\widehat{n})\otimes
N)\cong N$, which is simple. Hence
$\text{Hom}_{\mathbb{F}G^n}(V(\widehat{n}),M)\cong N$. Since $M$ as
an $\mathbb{F}G^n$-module is isomorphic to a direct sum of copies of
$V(\widehat{n})$, $M=V(\widehat{n})\otimes N$ by a dimensional
counting argument.
\end{proof}

\subsection{An equivalence of categories}

\begin{prop}\label{group tensor affine 4}
Let $M \in \HngMods$. Then
\begin{align*}
\Psi:V(\widehat{n})\otimes
\text{Hom}_{\mathbb{F}G^n}(V(\widehat{n}), I_{\widehat{n}}M) &
\longrightarrow
I_{\widehat{n}}M,\\
 v_1\otimes \cdots\otimes v_n\otimes\psi
& \mapsto \psi(v_1\otimes \cdots\otimes v_n)
\end{align*}
defines an isomorphism of $\mathcal{H}_{\widehat{n}}(G)$-modules.
\end{prop}

\begin{proof}
By Lemma~\ref{ealphaM}, $I_{\widehat{n}}M$ is an
$\mathcal{H}_{\widehat{n}}(G)$-module. It follows from
Propositions~\ref{Hom functor} and \ref{group tensor affine 3} that
$V(\widehat{n})\otimes \text{Hom}_{\mathbb{F}G^n} (V(\widehat{n}),
I_{\widehat{n}}M)$ is an $\mathcal{H}_{\widehat{n}}(G)$-module.

It can be easily checked that $\Psi$ is an
$\mathcal{H}_{\widehat{n}}(G)$-homomorphism. Since as an
$\mathbb{F}G^n$-module $I_{\widehat{n}}M$ is isomorphic to a direct
sum of copies of $V(\widehat{n})$, $\Psi$ is surjective and hence an
isomorphism by a dimension counting argument.
\end{proof}

We are now ready to prove the first main result of this paper.

\begin{thm}\label{morita equivalence} The functor
$\mathcal{F}:\mathcal{H}_n(G)$-\text{\bf mod}$^s\rightarrow
A_{n,r}$-\text{\bf mod} defined by
$$
\mathcal{F}(M) =\bigoplus_{\widehat{n}\in\mathcal{C}_n^r}
\text{Hom}_{\mathbb{F}G^n}(V(\widehat{n}),I_{\widehat{n}}M)
$$
is a category equivalence, with inverse $\mathcal{G}:
A_{n,r}$-\text{\bf mod}$\rightarrow \mathcal{H}_n(G)$-\text{\bf
mod}$^s$ given by
$$
\mathcal{G}(\oplus_{\widehat{n}\in\mathcal{C}_n^r} U_{\widehat{n}})
=\bigoplus_{\widehat{n}\in\mathcal{C}_n^r}
\text{ind}_{\mathcal{H}_{\widehat{n}}(G)}^{\mathcal{H}_n(G)}
(V(\widehat{n})\otimes U_{\widehat{n}}).
$$
\end{thm}

\begin{proof}
Note that the map $\Phi$ in Proposition \ref{group tensor affine 3}
is natural in $N$ and $\Psi$ in Proposition \ref{group tensor affine
4} is natural in $M$. Now using Lemma \ref{ealphaM},
Propositions~\ref{Hom functor}, \ref{group tensor affine 3} and
\ref{group tensor affine 4}, one easily checks that $\mathcal{F}
\mathcal{G} \cong \text{id}$ and $\mathcal{G}\mathcal{F} \cong
\text{id}.$
\end{proof}
\begin{rem}
Assume that $p$ does not divide the order of $G$. Then, every
finite dimensional $\mathcal{H}_n(G)$-module $M$ is semisimple
when restricted to $\mathbb{F}G^n$, and hence $\HngMods$ coincides
with $\Hng$-{\bf mod}. Moreover $c_k =|G|/d_k \in \mathbb I
-\{0\}$ for each $1\leq k\leq r$, and hence
$A_{\widehat{n},r}\cong \mathcal{H}_{\widehat{n}}$. Now
Theorem~\ref{morita equivalence} states that the wreath Hecke
algebra $\mathcal{H}_n(G)$ is Morita equivalent to the algebra
$\bigoplus_{\widehat{n}\in\mathcal{C}_n^r}
\mathcal{H}_{\widehat{n}}$.
\end{rem}

\section{Classification of simple modules and modular branching rules}
\label{branching rule}

We present two applications of the above category equivalence in
this section. We shall classify all finite dimensional simple
$\mathcal{H}_n(G)$-modules, and establish the modular branching
rule for the wreath Hecke algebra $\mathcal{H}_n(G)$ which
provides a description of the socle of the restriction to
$\mathcal{H}_{n-1,1}(G)$ of a simple $\Hng$-module.
\subsection{The simple $\Hno$-modules}

Denote by ${}^0\mathcal{H}_{\lambda}$ the subalgebra of $\Hno$ for
any composition $\lambda$ of $n$, in the same way as
$\mathcal{H}_{\widehat{n}} \subset \Hn$. It is well known that the
simple $\F S_n$-modules, denoted by $D^\mu$, are parameterized by
$p$-regular partitions $\mu$ of $n$ (cf. \cite[Corollary~
6.1.12]{JK}).
By letting each $x_i$ acting as a scalar $a\in \F$, we can extend an
$S_n$-module $W$ to an $\Hno$-module, which shall be denoted by
$W_a$. The classification of simple modules of the algebra $\Hno
=P_n \rtimes \F S_n$ over $\F$ is easily obtained by Clifford theory
as follows.

\begin{prop} \label{prop:Hn0}
Each simple $\Hno$-module is isomorphic to a module $
D_{\underline{a}, {\mu}}$ of the form
\begin{eqnarray} \label{eq:moduleD}
D_{\underline{a}, {\mu}} =
\text{ind}_{{}^0\mathcal{H}_{\widehat{\mu}}}^{\Hno} (
D^{\mu^1}_{a_1} \otimes \cdots\otimes D^{\mu^t}_{a_t})
\end{eqnarray}
where $a_1, \ldots, a_t$ are distinct scalars in $\F$, and $\mu^1,
\ldots, \mu^t$ are $p$-regular partitions such that $\widehat{\mu}
=(|\mu^1|,\ldots, |\mu^t|)$ is a composition of $n$ for some $t>0$.
Moreover, the above modules for varied $(a_i, \mu^i)$ form a
complete set of pairwise non-isomorphic simple $\Hno$-modules.
\end{prop}

\subsection{The simple $\Hng$-modules}

\begin{prop} \label{prop:cliff}
Suppose that $M$ is a simple $\mathcal{H}_n(G)$-module. Then, as an
$\F G^n$-module $M$ is semisimple.
\end{prop}

\begin{proof}
Take a simple $P_n(G)$-submodule $V (\underline{a})$ of $M$ which,
thanks to $P_n(G)\cong \mathbb{F}G^n\otimes P_n$, restricts to a
simple $\F G^n$-submodule $V\cong V_{i_1}\otimes\cdots\otimes
V_{i_n}$ with each $x_j$ acting as $a_j$ for $\underline{a} =(a_1,
\ldots, a_n) \in \mathbb F^n$.
It follows that $M_1 := \sum_{\pi \in S_n} \pi V$ is an
$\Hng$-submodule of $M$, and hence $M_1 =M$ since $M$ is simple.
Since each $\pi V$ is a simple $\F G^n$-module, $M$ as an $\F
G^n$-module is semisimple.
\end{proof}

\begin{cor} \label{same simple}
The categories $\HngMods$ and $\Hng$-{\bf mod} share the same class
of simple modules.
\end{cor}

\begin{thm} \label{IrrWHA2}
Each simple $\mathcal{H}_n(G)$-module is isomorphic to a module of
the form
\begin{eqnarray} \label{Hng simple}
D_{\widehat{n}} (L_\centerdot)
:= \text{ind}_{\mathcal{H}_{\widehat{n}}(G)}^{\mathcal{H}_n(G)}
(V_1^{\otimes n_1}\otimes L_1) \otimes \cdots\otimes (V_r^{\otimes
n_r}\otimes L_r)
\end{eqnarray}
where $\widehat{n}=(n_1, \ldots,n_r) \in\mathcal{C}_n^r$ and $L_k$
$(1\leq k\leq r)$ is a simple ${}^{c_k}\mathcal{H}_{n_k}$-module.
Moreover, the above modules for varied $\widehat{n}$ and $L_k$
($1\leq k\leq r$) form a complete set of pairwise non-isomorphic
simple $\mathcal{H}_n(G)$-modules.
\end{thm}

\begin{proof}
By Corollary~\ref{same simple}, each simple
$\mathcal{H}_n(G)$-module lies in the subcategory $\HngMods$. Now
the theorem follows by the category equivalence given in Theorem
\ref{morita equivalence}.
\end{proof}

\begin{rem}
Together with Proposition~\ref{prop:Hn0}, Theorem~\ref{IrrWHA2}
provides a complete classification of simple
$\mathcal{H}_n(G)$-modules.
\end{rem}

\subsection{Modular branching rules for $\mathcal{H}_n$}

Recall that the degenerate affine Hecke algebra $\mathcal{H}_n$ is
generated by $S_n$ and $y_1,\ldots,y_n.$ Let
$N\in\mathcal{H}_n$-{\bf mod} and $a\in\mathbb{F}$. Define
$\Delta_aN$ to be the generalized $a$-eigenspace of $y_n$ on $N$.
Since $y_n$ is central in the subalgebra
$\mathcal{H}_{n-1,1}\cong\mathcal{H}_{n-1}\otimes\mathcal{H}_1$ of
$\mathcal{H}_n$, $\Delta_aN$ is an $\mathcal{H}_{n-1,1}$-submodule
of $\text{res}_{\mathcal{H}_{n-1,1}}N$. Define
$$
e_a N
:=\text{res}^{\mathcal{H}_{n-1,1}}_{\mathcal{H}_{n-1}}(\Delta_aN).
$$
Then,
\begin{equation} \label{deg. aff. decomposition}
\text{res}_{\mathcal{H}_{n-1}}N =\bigoplus_{a\in\mathbb{F}}e_a N.
\end{equation}
Denote the socle of the $\mathcal H_{n-1}$-module $e_a N$ by
$$
\tilde{e}_aN:=\text{soc}(e_aN).
$$

The following modular branching rule for $\mathcal{H}_n$ is a
degenerate version of a result of Grojnowski-Vazirani \cite{GV}.

\begin{prop} \cite[Cor. 5.1.7, 5.1.8]{K2}\label{kle1}
Let $N$ be a simple $\mathcal{H}_n$-module and $a\in\mathbb{F}$.
Then either $\tilde{e}_aN=0$ or $\tilde{e}_aN$ is simple. Moreover,
the socle of $\text{res}_{\mathcal{H}_{n-1}}^{\mathcal{H}_n}N$ is
multiplicity-free.
\end{prop}
\subsection{Modular branching rules for $\Hno$}

As above, for $N\in\Hno$-{\bf mod} and $a\in\mathbb{F}$, the
generalized $a$-eigenspace of $y_n$ on $N$, denoted also by
$\Delta_a N$, is an ${}^0\mathcal{H}_{n-1,1}$-submodule of
$\text{res}_{{}^0\mathcal{H}_{n-1,1}}N$. Then,
$\text{res}_{{}^0\mathcal{H}_{n-1}}N =\oplus_{a\in\mathbb{F}}e_a N$
where $e_a N :=\text{res}_{{}^0\mathcal{H}_{n-1}} (\Delta_aN).$
Denote the socle of the ${}^0\mathcal H_{n-1}$-module $e_a N$ by $
\tilde{e}_aN:=\text{soc}(e_aN).$ We denote by
\begin{eqnarray} \label{composition}
\widehat{n}_i^- =(n_1,\ldots,n_i-1, \ldots, n_t),
 \quad
\widehat{n}_i^+ =(n_1,\ldots,n_i +1, \ldots, n_t)
\end{eqnarray}
the compositions of $n \mp 1$ associated to a composition
$\widehat{n} =(n_1, \ldots, n_t)$ of $n$ for $1 \le i \le t$. (It is
understood that the terms involving $\widehat{n}_i^-$ disappear for
those $i$ with $n_i =0$.)

The modular branching rules for $\Hno$ are described as follows.

\begin{prop} \label{rule:Hn0}
The socle of the restriction of a simple $\Hno$-module
$D_{\underline{a}, {\mu}}$ in \eqref{eq:moduleD} to
${}^0\mathcal{H}_{n-1}$ is multiplicity-free:
\begin{eqnarray*}
\text{soc} (\text{res}_{{}^0\mathcal{H}_{n-1}} D_{\underline{a},
{\mu}})
 \cong \bigoplus_{i=1}^t
\text{ind}_{{}^0\mathcal{H}_{\widehat{\mu}_i}}^{\Hno} \left(
D^{\mu^1}_{a_1} \otimes \cdots\otimes \text{soc}
(\text{res}_{S_{|\mu^i|-1}} D^{\mu^i})_{a_i} \otimes \cdots\otimes
D^{\mu^t}_{a_t} \right).
\end{eqnarray*}
Equivalently, $\tilde{e}_a (\text{res}_{{}^0\mathcal{H}_{n-1}}
D_{\underline{a}, {\mu}})=0$ unless $a =a_i$ for some $1\le i \le
t$, and
$$
\tilde{e}_{a_i} (\text{res}_{{}^0\mathcal{H}_{n-1}}
D_{\underline{a}, {\mu}}) \cong
\text{ind}_{{}^0\mathcal{H}_{\widehat{\mu}_i}}^{\Hno} \left(
D^{\mu^1}_{a_1} \otimes \cdots\otimes \text{soc}
(\text{res}_{S_{|\mu^i|-1}} D^{\mu^i})_{a_i} \otimes \cdots\otimes
D^{\mu^t}_{a_t} \right).
$$
\end{prop}

\begin{proof}
A version of Mackey Lemma gives us
\begin{eqnarray*}
\text{res}_{{}^0\mathcal{H}_{n-1}} D_{\underline{a}, {\mu}}
 \cong \bigoplus_{i=1}^t
\text{ind}_{{}^0\mathcal{H}_{\widehat{\mu}_i}}^{\Hno} \left(
D^{\mu^1}_{a_1} \otimes \cdots\otimes (\text{res}_{S_{|\mu^i|-1}}
D^{\mu^i})_{a_i} \otimes \cdots\otimes D^{\mu^t}_{a_t} \right).
\end{eqnarray*}
(See the proof of Lemma~\ref{lem:Mackey} below for a similar
argument.) Now the proposition follows from finding the socles of
both sides of the above isomorphism.
\end{proof}

\begin{rem} \label{rule:Sn}
According to \cite{K1} and \cite[Chapter 9]{K2}, for a $p$-regular
partition $\mu$ of $n$, the $S_{n-1}$-module $\text{soc}
(\text{res}_{S_{n-1}} D^{\mu})$ has an explicit multiplicity-free
decomposition according to the eigenvalues in $\mathbb I$ of the
$n$th Jucys-Murphy element.
\end{rem}

\subsection{Modular branching rules for $\mathcal{H}_n(G)$}

We start with a preparatory result.

\begin{prop}\label{Irr commutation}
Let $\widehat{n}=(n_1, \ldots,n_r) \in\mathcal{C}_n^r$ and $L_k$
$(1\leq k\leq r)$ be a ${}^{c_k}\mathcal{H}_{n_k}$-module. Then,
\begin{align}
&\text{ind}_{\mathcal{H}_{\widehat{n}}(G)}^{\mathcal{H}_n(G)}
(V_1^{\otimes n_1}\otimes L_1) \otimes\cdots\otimes (V_r^{\otimes
n_r}\otimes L_r)\notag \\
&\cong
\text{ind}_{\mathcal{H}_{\widehat{n}^{\tau}}(G)}^{\mathcal{H}_n(G)}
(V_{\tau(1)}^{\otimes n_{\tau(1)}}\otimes
L_{\tau(1)})\otimes\cdots\otimes (V_{ \tau(r)}^{\otimes
n_{\tau(r)}}\otimes L_{\tau(r)})\notag,
\end{align}
where $\widehat{n}^{\tau}=(n_{\tau(1)},\ldots,n_{\tau(r)})$ for any
$\tau\in S_r$.
\end{prop}

\begin{proof}
Let us denote the left-hand-side and the right-hand side of the
isomorphism in the Proposition by $\texttt L$ and $\texttt R$
respectively. By Theorem~\ref{morita equivalence}, it suffices to
show that $\mathcal F(\texttt{L}) \cong \mathcal F(\texttt{R})$.
%
Indeed, for $\widehat{n}\neq \widehat{m}\in\mathcal{C}_n^r$,
$\text{Hom}_{\mathbb{F}G^n}(V(\widehat{m}),I_{\widehat{m}}\texttt{L})
=
\text{Hom}_{\mathbb{F}G^n}(V(\widehat{m}),I_{\widehat{m}}\texttt{R})
=0$ (actually $I_{\widehat{m}}\texttt{L} = I_{\widehat{m}}\texttt{R}
=0$.) Also,
$\text{Hom}_{\mathbb{F}G^n}(V(\widehat{n}),I_{\widehat{n}}\texttt{L})
\cong L_1 \otimes \cdots \otimes L_r \cong
\text{Hom}_{\mathbb{F}G^n}(V(\widehat{n}),I_{\widehat{n}}\texttt{R}).$
This proves the proposition.
\end{proof}

Let us denote by $\mathcal{H}_{n-1,1}(G)$ the subalgebra of
$\mathcal{H}_n(G)$ generated by $G^n, x_1,\ldots, x_n$ and
$S_{n-1}$. Then $\mathcal{H}_{n-1,1}(G)\cong
\mathcal{H}_{n-1}(G)\otimes\mathcal{H}_1(G)$. We shall extend the
notation for $D_{\widehat{n}}(L_\centerdot)$, see \eqref{Hng
simple}, to the case when $L_k$ are not necessarily simple modules.
The following can be regarded as a variant of Mackey's lemma.

\begin{lem}  \label{lem:Mackey}
Let $\widehat{n}=(n_1, \ldots,n_r) \in\mathcal{C}_n^r$ and $L_k$
$(1\leq k\leq r)$ be a ${}^{c_k}\mathcal{H}_{n_k}$-module. Then,
\begin{align}  \label{eq:Mackey} \text{res}_{\mathcal{H}_{n-1,1}(G)}
D_{\widehat{n}}(L_\centerdot)
 \cong
\bigoplus_{a\in\mathbb{F}, 1\leq k\leq r}
D_{\widehat{n}_k^-}({e}_a L_\centerdot) \otimes(V_k\otimes L(a)),
\end{align}
where $D_{\widehat{n}_k^-}({e}_a L_\centerdot)$ denotes the
$\mathcal{H}_{n-1}(G)$-module
$$
\text{ind}_{\mathcal{H}_{\widehat{n}_k^-}(G)}^{\mathcal{H}_{n-1}(G)}
\left( (V_1^{\otimes n_1}\otimes L_1)\otimes\cdots\otimes
(V_k^{\otimes n_k-1}\otimes {e}_aL_k)\otimes\cdots\otimes
(V_r^{\otimes n_r}\otimes L_r) \right).
$$
\end{lem}

\begin{proof}
It can be easily checked that $D_{\widehat{n}_r^-}({e}_a
L_\centerdot)\otimes (V_r\otimes L(a))$ is a
$\mathcal{H}_{n-1,1}(G)$-submodule of
$\text{res}_{\mathcal{H}_{n-1,1}(G)}
D_{\widehat{n}}(L_\centerdot)$ for all $a\in\mathbb{F}$ by
Mackey's Lemma. (It is understood that we take the biggest $k$
satisfying $n_k\neq 0$ if $n_r=0$.) Then Proposition~\ref{Irr
commutation} implies that $D_{\widehat{n}_k^-}({e}_a L_\centerdot)
\otimes(V_k\otimes L(a))$ is $\mathcal{H}_{n-1,1}(G)$-submodule of
$\text{res}_{\mathcal{H}_{n-1,1}(G)}
D_{\widehat{n}}(L_\centerdot)$ for each $a\in\mathbb{F}$ and
$1\leq k\leq r$, and hence we have
$$
\sum_{a\in\mathbb{F}, 1\leq k\leq r} D_{\widehat{n}_k^-}({e}_a
L_\centerdot) \otimes(V_k\otimes
L(a))\subseteq\text{res}_{\mathcal{H}_{n-1,1}(G)}
D_{\widehat{n}}(L_\centerdot).
$$
Since $V_k\otimes L(a)$ are pairwise non-isomorphic simple
$\mathcal{H}_1(G)$-modules for distinct $(k,a)$, the above sum is a
direct sum and then \eqref{eq:Mackey} follows by a dimensional
counting argument.
\end{proof}

We are now ready to establish the modular branching rules for
$\Hng$.

\begin{thm} \label{rule:Hng}
Retain the notation of the simple $\mathcal{H}_n(G)$-module
$D_{\widehat{n}}(L_\centerdot)$ in \eqref{Hng simple}. Then,
\begin{align*}
\text{soc}\left(\text{res}_{\mathcal{H}_{n-1,1}(G)}
D_{\widehat{n}}(L_\centerdot)\right)
 \cong
\bigoplus_{a\in\mathbb{F}, 1\leq k\leq r}
D_{\widehat{n}_k^-}(\tilde{e}_a L_\centerdot) \otimes(V_k\otimes
L(a)),
\end{align*}
where $D_{\widehat{n}_k^-}(\tilde{e}_a L_\centerdot)$ denotes the
semisimple $\mathcal{H}_{n-1}(G)$-module
$$\text{ind}_{\mathcal{H}_{\widehat{n}_k^-}(G)}^{\mathcal{H}_{n-1}(G)}
\left((V_1^{\otimes n_1}\otimes L_1)\otimes\cdots\otimes
(V_k^{\otimes n_k-1}\otimes \tilde{e}_aL_k)\otimes\cdots\otimes
(V_r^{\otimes n_r}\otimes L_r) \right).
$$
\end{thm}

\begin{proof}
Follows from Lemma~\ref{lem:Mackey} by observing that the socle of
the $\mathcal{H}_{n-1}(G)$-module $D_{\widehat{n}_k^-}({e}_a
L_\centerdot)$ is $D_{\widehat{n}_k^-}(\tilde{e}_a L_\centerdot)$.
\end{proof}

\subsection{Modular branching rules for $G_n$}

Let us denote by $G_{\widehat{n}} = G_{n_1} \times \cdots \times
G_{n_r}$ the subgroups of the wreath product $G_n$ for $\widehat{n}
=(n_1,\ldots, n_r) \in \Cnr$.
Recall \cite[Theorem~ 4.3.34]{JK} that a simple $G_n$-module is
isomorphic to
\begin{eqnarray*}
D_{\widehat{n}}^{\mu^\centerdot} =
\text{ind}^{G_n}_{G_{\widehat{n}}} \left((V_1^{\otimes n_1}\otimes
 D^{\mu^1})\otimes\cdots \otimes
(V_r^{\otimes n_r}\otimes D^{\mu^r}) \right)
\end{eqnarray*}
where $\mu^1,\ldots, \mu^r$ are $p$-regular partitions with
$\widehat{n} =(|\mu^1|,\ldots, |\mu^r|) \in \Cnr$. The proof of the
following modular branching rules for wreath products is similar to
Proposition~\ref{rule:Hn0} and Theorem~\ref{rule:Hng} and will be
skipped.

\begin{prop} \label{rule:Gn}
Retain notations above. Then the socle of $\text{res}_{G_{n-1,1}}
D_{\widehat{n}}^{\mu^\centerdot}$ is isomorphic to
 $
\bigoplus_{1\leq k\leq r} D_{\widehat{n}_k^-}^{\text{soc}
(\mu^\centerdot)} \otimes V_k,
 $
where $D_{\widehat{n}_k^-}^{\text{soc} (\mu^\centerdot)}$ denotes
the semisimple $G_{n-1}$-module
$$
\text{ind}_{G_{\widehat{n}_k^-}}^{G_{n-1}} \left((V_1^{\otimes
n_1}\otimes D^{\mu^1})\otimes\cdots\otimes (V_k^{\otimes n_k-1}
\otimes \text{soc} (\text{res}_{S_{n_k-1}}D^{\mu^k}))
\otimes\cdots\otimes (V_r^{\otimes n_r}\otimes D^{\mu^r}) \right).
$$

\end{prop}

\section{Cyclotomic wreath Hecke algebras and crystals}\label{crystal operators}

\subsection{A block decomposition}


We shall construct a decomposition of a module $M$ in $\HngMods$,
similar to \cite[Section 4.1, 4.2]{K2}. For any
$\underline{a}=(a_1,\ldots,a_n)\in\mathbb{F}^n$, let
$M_{\underline{a}}$ be the simultaneous generalized eigenspace of
$M$ for the commuting operators $x_1,\ldots,x_n$ with eigenvalues
$a_1,\ldots,a_n$. Then as a $P_n$-module, we have
$$M=\bigoplus_{\underline{a}\in\mathbb{F}^n}M_{\underline{a}}.$$

A given $\underline{a}\in\mathbb{F}^n$ defines a 1-dimensional
representation of the algebra $\Lambda_n$ of symmetric polynomials
in $x_1,\ldots,x_n$:
$$
\omega_{\underline{a}}:\Lambda_n\rightarrow \mathbb{F}, \quad
f(x_1,\ldots,x_n)=f(a_1,\ldots,a_n).
$$
Write $\underline{a}\sim\underline{b}$ if they lie in the same
$S_n$-orbit. Observe that $\underline{a}\sim\underline{b}$ if and
only if $\omega_{\underline{a}}=\omega_{\underline{b}}$. For each
orbit $\gamma\in\mathbb{F}^n/\sim$, we set
$\omega_{\gamma}:=\omega_{\underline{a}}$ for any
$\underline{a}\in\gamma$. Let
$$
M[\gamma] =\{v\in M| (z-\omega_{\gamma}(z))^Nv =0 \text{ for all }
z\in\Lambda_n \text{ and } N\gg 0 \}.
$$
Then we have
$$
M[\gamma]=\bigoplus_{\underline{a}\in\gamma}M_{\underline{a}}.
$$
Since $\Lambda_n$ is contained in the center of $\mathcal{H}_n(G)$
by Theorem \ref{center}, $M[\gamma]$ is an
$\mathcal{H}_n(G)$-module and we have the following decomposition
in $\HngMods$:
\begin{equation} \label{decomposition 2}
M=\bigoplus_{\gamma\in\mathbb{F}^n/\sim}M[\gamma].
\end{equation}
By (\ref{decomposition 2}) and recalling the decomposition in
Lemma~\ref{Malpha}, we define, for $\n \in \Cnr \text{ and
}\gamma\in\mathbb{F}^n/\sim$, that
$$
M[\n,\gamma]:=M_{\n} \cap M[\gamma].
$$
Since $x_1,\ldots,x_n$ commute with $G^n$, it follows that
$M[\n,\gamma]=(M_{\n})[\gamma]=(M[\gamma])_{\n}.$ Then we have the
following decomposition in $\HngMods$:
\begin{eqnarray} \label{blocks}
M=\bigoplus_{\n \in \Cnr, \gamma\in\mathbb{F}^n/\sim}M[\n,\gamma].
\end{eqnarray}
This gives us a block decomposition of $\HngMods$ by applying
Theorem~\ref{morita equivalence} and the well known block
decomposition for $\Hn$ (and a simpler one for $\Hno$).

\subsection{The cyclotomic wreath Hecke algebras}
Recall scalars $c_k \in \F$ from Lemma~\ref{weig-t-0}. {\bf In the
remainder of this section we assume that $c_k\in\mathbb{I}$ for
all $1\leq k\leq r$} (see however Remark~\ref{rem:general p} on
the general case). Note that the assumption is automatically
satisfied if $p$ does not divide the order of $G$, since $c_k=
|G|/d_k$ by Lemma~\ref{weig-t-0}.

An $\mathcal{H}_n(G)$-module is called {\em integral} if it is
finite dimensional and all eigenvalues of $x_1,\ldots,x_n$ on $M$
belong to $\mathbb{I}$. Denote by $\mathcal{H}_n(G)$-{\bf
mod}$_{\mathbb{I}}$ the full subcategory of $\HngMods$ consisting
of all integral $\Hng$-modules which are semisimple over the
subalgebra $\F G^n$.  It is known \cite[Section 7.1]{K2} that the
study of simple modules for the degenerate affine Hecke algebra
$\Hn$ reduces to those of integral simple $\Hn$-modules (the case
of $\Hno$ is similar and simpler). Then by Theorem~\ref{morita
equivalence} and Corollary~\ref{same simple}, to study simple
$\mathcal{H}_n(G)$-modules, it suffices to study simple objects in
$\mathcal{H}_n(G)$-{\bf mod}$_{\mathbb{I}}$.

Introduce the {\em intertwining elements}:
$$\Omega_i:=s_i(x_i-x_{i+1})+t_{i,i+1}, \qquad 1\leq i\leq n-1.$$

\begin{lem}
The following identities hold in $\mathcal{H}_n(G)$:
\begin{align}
\Omega_i^2=t_{i,i+1}^2-(x_i&-x_{i+1})^2\label{intertwine 1}\\
\Omega_ix_i=x_{i+1}\Omega_i,\quad
\Omega_ix_{i+1}&=x_i\Omega_i,\quad \Omega_ix_j=x_j\Omega_i, \quad
j\neq i,i+1\label{intertwine 2}.
\end{align}
\end{lem}
\begin{proof}
By (\ref{symmetric and polnomial 2}) and (\ref{s and t}), we
calculate that
\begin{align}
\Omega^2_i &=
s_i(x_i-x_{i+1})s_i(x_i-x_{i+1}) +2t_{i,i+1}s_i(x_i-x_{i+1})+t_{i,i+1}^2\notag\\
&=(x_{i+1}-x_i-2s_it_{i,i+1})(x_i-x_{i+1}) +2t_{i,i+1}s_i(x_i-x_{i+1})+t_{i,i+1}^2\notag\\
&=t_{i,i+1}^2 -(x_i-x_{i+1})^2. \notag
\end{align}
Also,
\begin{align}
\Omega_ix_i&=s_ix_i(x_i-x_{i+1})+x_it_{i,i+1}
=x_{i+1}s_i(x_i-x_{i+1})-t_{i,i+1}(x_i-x_{i+1})+x_it_{i,i+1}\notag\\
&=x_{i+1}s_i(x_i-x_{i+1})+t_{i,i+1}x_{i+1}=x_{i+1}\Omega_i\notag.
\end{align}
Similarly, we have $\Omega_ix_{i+1}=x_i\Omega_i$. For $j\neq
i,i+1$, $s_ix_j=x_js_i$, and hence $\Omega_ix_j=x_j\Omega_i$.
\end{proof}

\begin{lem}\label{integral1}
Let $M\in\mathcal{H}_n(G)$-{\bf mod}${}^s$ and fix $j$ with $1\leq
j\leq n$. Assume that all eigenvalues of $x_j$ on $M$ belong to
$\mathbb{I}$. Then $M$ is integral.
\end{lem}
\begin{proof} It suffices to show that the eigenvalues of $x_k$
belong to $\mathbb{I}$ if and only if the eigenvalues of $x_{k+1}$
belong to $\mathbb{I}$, for $1\leq k\leq n-1.$ By
Corollary~\ref{weig-t-1}, Lemma~\ref{ealphaM} and
Lemma~\ref{Malpha}, it is enough to focus on the subspaces
$I_{\n}M$ for all $\n=(n_1,\ldots,n_r)\in \Cnr$. Assume that all
eigenvalues of $x_k$ on $I_{\n}M$ belong to $\mathbb{I}$. Let $a$
be an eigenvalue for the action of $x_{k+1}$ on $I_{\n}M$. Since
$x_k$ and $x_{k+1}$ commute, we can pick $v$ lying in the
$a$-eigenspace of $x_{k+1}$ so that $v$ is also an eigenvector of
$x_k$, of eigenvalue $b$. By assumption we have $b\in\mathbb{I}.$
By (\ref{intertwine 2}), we have
$x_{k}\Omega_k=\Omega_{k}x_{k+1}$. So if $\Omega_kv\neq 0$, then
$x_k\Omega_k v=a\Omega_k v$, hence $a$ is an eigenvalue of $x_k$,
and so $a\in\mathbb{I}$. Else, $\Omega_kv=0$, then applying
(\ref{intertwine 1}), we have $(b-a)^2v=t_{k,k+1}^2v$. Since
$I_{\n}M$ is isomorphic to the direct sum of copies $V_1^{\otimes
n_1}\otimes\cdots\otimes V_r^{\otimes n_r}$, by Corollary
\ref{weig-t-1}, $t_k^2v=0$ or $t_k^2v=c_{l_k}^2v$. Thus $a=b$ or
$a=b\pm c_{l_k}$. Similarly, we can show that all eigenvalues of
$x_k$ on $I_{\n}M$ belong to $\mathbb{I}$ if assuming all
eigenvalues of $x_{k+1}$ on $I_{\n}M$ belong to $\mathbb{I}$.
\end{proof}

Set
$$
\Delta=\{\lambda=(\lambda_i)_{i\in\mathbb{I}}|\lambda_i\in\mathbb{Z}_+,
\text{ and only finitely many }\lambda_i \text{ are nonzero}\}.
$$
Let
\begin{eqnarray}   \label{poly g}
g_{\lambda} \equiv g_{\lambda}(x_1) =\prod_{i\in \mathbb I} (x_1
-i)^{\lambda_i}.
\end{eqnarray}
The {\em cyclotomic wreath Hecke algebra} is defined to be the
quotient algebra by the two-sided ideal $J_\lambda$ of
$\mathcal{H}_n(G)$ generated by $g_{\lambda}$:
\begin{eqnarray}  \label{cyclot}
\mathcal{H}_n^{\lambda}(G) =\mathcal{H}_n(G)/J_\lambda, \qquad
\lambda\in\Delta.
\end{eqnarray}

\begin{rem} \label{coincidence between cyclotomics}In
the case $G=\{1\}$, $\mathcal{H}_n^{\lambda}(G)$ coincides with
degenerate cyclotomic Hecke algebra $\mathcal{H}_n^{\lambda}$ (see
\cite[Section 7.3]{K2}).
\end{rem}

\begin{lem}\label{integral2}
Let $M\in\mathcal{H}_n(G)$-{\bf mod}${}^s$. Then $M$ is integral
if and only if $J_{\lambda}M=0$ for some $\lambda\in\Delta$.
\end{lem}

\begin{proof}
If $J_{\lambda}M=0$, then the eigenvalues of $x_1$ on $M$ are all
in $\mathbb{I}$, and hence $M$ is integral in view of Lemma
\ref{integral1}. Conversely, suppose $M$ is integral. Then the
minimal polynomial of $x_1$ on $M$ is of the form $\prod_{i\in
\mathbb I} (t-i)^{\lambda_i}$ for some $\lambda_i\in\mathbb{Z}_+$.
So if we set $J_{\lambda}$ to be the two-sided ideal of
$\mathcal{H}_n(G)$ generated by $\prod_{i\in \mathbb I}
(x_1-i)^{\lambda_i}$, we certainly have that $J_{\lambda}M=0$.
\end{proof}
We shall denote by $\mathcal{H}_n^{\lambda}(G)$-{\bf mod}$^s$ the
full subcategory of $\mathcal{H}_n^{\lambda}(G)$-{\bf mod}
consisting of finite-dimensional $\Hng$-modules which are
semisimple when restricted to $\F G^n$. By inflation along the
canonical homomorphism $\mathcal{H}_n(G)\rightarrow
\mathcal{H}_n^{\lambda}(G)$, we can identify
$\mathcal{H}_n^{\lambda}(G)$-{\bf mod} (resp.
$\mathcal{H}_n^{\lambda}(G)$-{\bf mod}${}^s$) with the full
subcategory of $\mathcal{H}_n(G)$-{\bf mod} (resp. $\HngMods$)
consisting of all modules $M$ with $J_{\lambda}M=0$. By
Lemma~\ref{integral2}, to study modules in the category
$\HngModI$, we may instead study modules in the categories
$\mathcal{H}_n^{\lambda}(G)$-{\bf mod}${}^s$ for all
$\lambda\in\Delta$.

Our subalgebra $\mathbb{F}G_n$ of $\mathcal{H}_n(G)$ plays an
analogous role as the subalgebra $\mathbb{F}S_n$ of
$\mathcal{H}_n$. It turns out that all the lemmas in \cite[Section
7.5]{K2} used to prove the PBW basis theorem for the degenerate
cyclotomic Hecke algebras remain valid if we replace
$\mathbb{F}S_n$ there by $\mathbb{F}G_n$. So we have the
following.

\begin{prop}\label{PBW basis for cyclotomic algebras}
Let $d=\sum_{i\in\mathbb{I}}\lambda_i$. The elements
$$
\{x^{\alpha}\pi g \mid \alpha\in\mathbb{Z}^n_+ \text{ with
}\alpha_1,\ldots,\alpha_n<d, \pi\in S_n, g\in G^n\}
$$
form a basis for $\mathcal{H}_n^{\lambda}(G)$.
\end{prop}

\begin{rem} \label{minimal quotient}
For nonzero $\lambda$, $\F G_n$ is a subalgebra of
$\mathcal{H}_n^{\lambda}(G)$. In particular, for
$\Lambda_0=(\lambda_i)_{i\in\mathbb{I}}$ with $\lambda_0=1$ and
$\lambda_i=0$ for $i\neq 0$, we have
$\mathcal{H}_n^{\Lambda_0}(G)\cong\mathbb{F}G_n$.
\end{rem}
\begin{cor}
The subalgebra of $\mathcal{H}_{n}^{\lambda}(G)$ generated by
$x_1,\ldots,x_{n-1}, \pi\in S_{n-1}, g\in G^{n}$ is isomorphic to
$\mathcal{H}_{n-1}^{\lambda}(G) \times G$.
\end{cor}
\subsection{The functors $e_{i,\chi^k}^{\lambda}$ and
$f_{i,\chi^k}^{\lambda}$.} In view of (\ref{blocks}), we have the
following decomposition in $\mathcal{H}_n(G)$-{\bf
mod}$_{\mathbb{I}}$:
$$
M=\bigoplus_{\n \in \Cnr,\gamma\in\mathbb{I}^n/\sim} M[\n,\gamma].
$$

Set $\Gamma_n$ to be the set of non-negative integral linear
combinations $\gamma=\sum_{i\in\mathbb{I}}\gamma_i\varepsilon_i$
of the standard basis $\varepsilon_i$ of
$\mathbb{Z}^{|\mathbb{I}|}$ such that
$\sum_{i\in\mathbb{I}}\gamma_i=n$. If
$\underline{a}\in\mathbb{I}^n$, define its content to be
$$
\text{cont}(\underline{a})=\sum_{i\in\mathbb{I}}
\gamma_i\varepsilon_i \in \Gamma_n, \quad \text{where }\gamma_i=\#
\{j=1,\ldots,n|a_j=i\}.
$$
The content function induces a canonical bijection between
$\mathbb{I}^n/\sim$ and $\Gamma_n$, and we will identify the two
sets. Now the above decomposition in $\mathcal{H}_n(G)$-{\bf
mod}$_{\mathbb{I}}$ can be written as
\begin{equation} \label{eq:decomp}
M=\bigoplus_{\n \in \Cnr,\gamma\in\Gamma_n}M[\n,\gamma].
\end{equation}
Such a decomposition also makes sense in the category
$\mathcal{H}_n^{\lambda}(G)$-{\bf mod}${}^s$.

\begin{defn}\label{crystal operators defn}
Suppose that $M\in\mathcal{H}_n^{\lambda}(G)$-{\bf mod}${}^s$ and
that $M=M[\n,\gamma]$ for some $\n \in \Cnr$ and
$\gamma\in\Gamma_n$. We define (see \eqref{composition} for
notations)
\begin{align*}
e_{i,\chi^k}^{\lambda}M &= \left \{
\begin{array}{ll}
\text{Hom}_G(V_k, \text{res}_{\mathcal{H}_{n-1}^{\lambda}(G)
\times G}M)[\n_k^-,\gamma-\varepsilon_{ic_k}], & \text{ if
 } c_{k}\neq 0
 \\
\text{Hom}_G(V_k, \text{res}_{\mathcal{H}_{n-1}^{\lambda}(G)
\times G}M)[\n_k^-,\gamma-\varepsilon_{i}],& \text{ if
 } c_{k}=0,
 \end{array}
 \right.
 \\
 f_{i,\chi^k}^{\lambda}M &= \left \{
 \begin{array}{ll}
(\text{ind}^{\mathcal{H}_{n+1}^{\lambda}(G)}_{\mathcal{H}_n^{\lambda}(G)
\times G}(M\otimes V_k)) [\n_k^+,\gamma+\varepsilon_{ic_k}], &
\text{ if
 } c_{k}\neq 0
  \\
(\text{ind}^{\mathcal{H}_{n+1}^{\lambda}(G)}_{\mathcal{H}_n^{\lambda}(G)
\times G}(M\otimes V_k)) [\n_k^+,\gamma+\varepsilon_{i}], & \text{
if
 } c_{k}=0.
  \end{array}
 \right.
\end{align*}

We extend $e_{i,\chi^k}^{\lambda}$ (resp.
$f_{i,\chi^k}^{\lambda}$) to functors from
$\mathcal{H}_n^{\lambda}(G)$-{\bf mod}${}^s$ to
$\mathcal{H}_{n-1}^{\lambda}(G)$-{\bf mod}${}^s$(resp. from
$\mathcal{H}_n^{\lambda}(G)$-{\bf mod}${}^s$ to
$\mathcal{H}_{n+1}^{\lambda}(G)$-{\bf mod}${}^s$) by the direct
sum decomposition~\eqref{eq:decomp}.
\end{defn}

\begin{rem}
If $G=\{1\}$ is the trivial group, the functors
$e_{i,\chi^k}^{\lambda}$ and $f_{i,\chi^k}^{\lambda}$ (with the
index $\chi^k$ dropped) coincide with the ones $e_i^{\lambda}$ and
$f_i^{\lambda}$ defined in \cite[Section 8.1]{K1}.
\end{rem}

\subsection{An equivalence of categories.}
Let $S_{n-1}^{\prime}$ be the subgroup of $S_n$ generated by
$s_2,\ldots,s_{n-1}$. The following lemma follows from
\cite[Proposition A.3.2]{Ze} which describes the double cosets
$S_{n-1}^{\prime}\setminus S_n /S_{\widehat{n}}$. For each
$\widehat{n}=(n_1,\ldots,n_r)\in\mathcal{C}_n^r$ and $1\leq k\leq
r$, set
$$
\widehat{n}_{1\cdot\cdot k}=n_1+\cdots+n_k.
$$

\begin{lem}\label{left cosets representatives}
Retain the above notations. Then there exists a complete set
$\Theta(\widehat{n})$ of representatives of left cosets of
$S_{\widehat{n}}$ in $S_n$ such that any $w\in\Theta(\widehat{n})$
is of the form $\sigma (1, \widehat{n}_{1\cdot\cdot k}+1)$ for
some $\sigma\in S_{n-1}^{\prime}$ and $0\leq k\leq r-1$. (It is
understood that $(1,\widehat{n}_{1\cdot\cdot k}+1)=1$ when $k=0$.)
\end{lem}

Note that $(1,m+1) =s_{m}\cdots s_2s_1s_2\cdots s_{m}$. The next
lemma follows from (\ref{symmetric and polnomial 2}) and the
identity $t_{i,j}s_j=s_jt_{i,j+1}$ for $1\leq i<j\leq n-1$ in
$\mathcal{H}_n(G).$

\begin{lem} \label{relation between t and s}
The following equation holds in $\mathcal{H}_n(G)$ for $0\leq
k\leq r-1$:
\begin{align}
x_1(1,\widehat{n}_{1\cdot\cdot k}+1)
&=(1,\widehat{n}_{1\cdot\cdot k}+1) x_{\widehat{n}_{1\cdot\cdot k}+1}\notag\\
&-\sum_{l=1}^{\widehat{n}_{1\cdot\cdot
k}}s_{\widehat{n}_{1\cdot\cdot k}} \cdots
s_2s_1s_2\cdots\widehat{s_l}\cdots s_{\widehat{n}_{1\cdot\cdot
k}}t_{l,\widehat{n}_{1\cdot\cdot k}+1}.\notag
\end{align}
\end{lem}

Now assume that $p>0$. Let $\{\alpha_i|i\in\mathbb{I}\}$ be the
simple roots of the complex affine Lie algebra $\widehat{sl}_p$
and $\{h_i|i\in\mathbb{I}\}$ be the corresponding simple coroots.
Let $P_+$ be the set of all dominant integral weights. Recall in
\cite[Section 8.1]{K2} for each $\mu\in P_+$, the degenerate
cyclotomic Hecke algebra is
$$
\mathcal{H}_n^{\mu}=\mathcal{H}_n/ \langle
\prod_{i\in\mathbb{I}}(y_1-i)^{\langle h_i,\mu\rangle} \rangle.
$$

For $\lambda\in \Delta$ and $1\leq k\leq r$, define $\lambda[k]\in
P_+$ by letting
$$
\langle h_i,\lambda[k]\rangle =\lambda_{ic_k}, \quad \forall
i\in\mathbb{I}.
$$
Further denote the algebra
\begin{eqnarray} \label{Anrlam}
A_{n,r}^\lambda
=\bigoplus_{\widehat{n}=(n_1,\ldots,n_r)\in\mathcal{C}_n^r}
\mathcal{H}_{n_1}^{\lambda[1]}
\otimes\cdots\otimes\mathcal{H}_{n_r}^{\lambda[r]}.
\end{eqnarray}

\begin{thm} \label{morita cyclotomic}
Assume that $p> 0$ and $p$ does not divide $|G|.$ Then the functor
$\mathcal{F}$ in Theorem \ref{morita equivalence} induces a
category equivalence $\mathcal{F}^{\lambda}:
\mathcal{H}_n^{\lambda}(G)\text{-{\bf
 mod}} \longrightarrow A_{n,r}^\lambda$-{\bf
mod}.
\end{thm}

\begin{proof}
Recall the definition of $g_\lambda$ and $J_\lambda$, see
\eqref{cyclot}. The category $\mathcal{H}_n^{\lambda}(G)$-{\bf
mod} can be identified with the full subcategory of
$\mathcal{H}_n(G)$-{\bf mod} consisting of all modules $M$ with
$J_{\lambda}M=0$. In view of Lemma \ref{Malpha}, $J_{\lambda}M=0$
if and only if $J_{\lambda}M_{\widehat{n}}=0$ for each
$\widehat{n}\in\mathcal{C}_n^r$. By Lemma \ref{ealphaM} and
Proposition \ref{group tensor affine 4}, we have
$$
M_{\widehat{n}}=\text{ind}_{\mathcal{H}_{\widehat{n}}(G)}^{\mathcal{H}_n(G)}
I_{\widehat{n}}M, \qquad I_{\widehat{n}}M\cong
V(\widehat{n})\otimes_{\mathbb{F}}\text{Hom}_{\mathbb{F}G^n}(V(\widehat{n}),
I_{\widehat{n}}M).
$$
As vector spaces, we have
$$
M_{\widehat{n}}=\bigoplus_{w\in\Theta(\widehat{n})}w\otimes
I_{\widehat{n}}M.
$$
By Lemma \ref{left cosets representatives}, for each
$w\in\Theta(\widehat{n})$, there exists $\sigma\in
S_{n-1}^{\prime}$ such that $w=\sigma (1,\widehat{n}_{1\cdot\cdot
k}+1)$ for some $0\leq k\leq r-1$. So
$g_{\lambda}w=g_{\lambda}\sigma (1,\widehat{n}_{1\cdot\cdot
k}+1)=\sigma g_{\lambda}(1,\widehat{n}_{1\cdot\cdot k}+1).$ Note
that $t_{l, \widehat{n}_{1\cdot\cdot k}+1}=0$ on $
I_{\widehat{n}}M$ for $1\leq l\leq \widehat{n}_{1\cdot\cdot k}$,
so
$$
x_1(1,\widehat{n}_{1\cdot\cdot k}+1)\otimes
z=(1,\widehat{n}_{1\cdot\cdot k}+1)\otimes
x_{\widehat{n}_{1\cdot\cdot k}+1}z
$$
for $z\in I_{\widehat{n}}M$ by Lemma~\ref{relation between t and
s}, and thus $g_{\lambda}w\otimes z=\sigma
(1,\widehat{n}_{1\cdot\cdot k}+1)\otimes g_{\lambda,k}z$, where
$$g_{\lambda,k}:=\prod_{i\in\mathbb{I}}(x_{\widehat{n}_{1\cdot\cdot k}+1}-i)^{\lambda_i}.$$
Therefore $g_{\lambda}M_{\widehat{n}}=0$ if and only if
$g_{\lambda,k}I_{\widehat{n}}M=0$ for $0\leq k\leq r-1$. By
Propositions~\ref{group tensor affine 3} and \ref{group tensor
affine 4}, $g_{\lambda,k}$ acts as zero on $I_{\widehat{n}}M$ if
and only if $\prod_{i\in\mathbb{I}}(c_ky_{\widehat{n}_{1\cdot\cdot
k}+1}-i)^{\lambda_i}$ acts as zero on
$\text{Hom}_{\mathbb{F}G^n}(V(\widehat{n}), I_{\widehat{n}}M)$,
that is, $\prod_{i\in\mathbb{I}}(y_{\widehat{n}_{1\cdot\cdot
k}+1}-i)^{\langle h_i,\lambda[k]\rangle}$ acts on zero on
$\text{Hom}_{\mathbb{F}G^n}(V(\widehat{n}), I_{\widehat{n}}M)$
since $\frac{1}{c_k}\mathbb{I}=\mathbb{I}$ if $p>0$. Therefore
$g_{\lambda}M=0 $ if and only if
$\text{Hom}_{\mathbb{F}G^n}(V(\widehat{n}), I_{\widehat{n}}M)\in
A_{n,r}^\lambda\text{-{\bf mod}} $ for each $
\widehat{n}\in\mathcal{C}_n^r$ as desired.
\end{proof}

\begin{rem} \label{block}
The blocks of the degenerate cyclotomic Hecke algebras are
classified by the $S_n$-orbits of the $n$-tuple eigenvalues of
$x_1,\ldots, x_n$ \cite{Br}. By the Morita equivalence in
Theorem~\ref{morita cyclotomic}, \eqref{eq:decomp} provides us a
block decomposition in $\mathcal{H}_n^{\lambda}(G)$-{\bf mod} when
$p$ does not divide $|G|$.
\end{rem}

\begin{rem} \label{rem:general p}
The assumption that $p$ does not divide $|G|$ in
Theorem~\ref{morita cyclotomic} is imposed merely for avoiding
complicated notations. We can drop it and also the assumption that
$c_k(1\leq k\leq r)$ are integral (compare Theorem~\ref{morita
equivalence}) with the same proof, if we replace
$\mathcal{H}_n^{\lambda}$-{\bf mod} by
$\mathcal{H}_n^{\lambda}$-{\bf mod}${}^s$ and modify suitably the
definition  \eqref{Anrlam} of the algebra $A_{n,r}^\lambda$. The
modified algebra $A_{n,r}^\lambda$ might admit {\em non-integral}
degenerate cyclotomic Hecke algebras as its tensor factors, since
now possibly $c_k \not \in \mathbb I$ (non-integral simple modules
and modular branching rules of degenerate affine or cyclotomic
Hecke algebras can be reduced to integral cases, cf. \cite[Section
7.1]{K2}). In addition, some quotient algebras of $\Hno$
(corresponding to the cases when $c_k =0$) will appear as tensor
factors of $A_{n,r}^\lambda$.

The remaining case when $p=0$ can also be handled similarly with
somewhat more involved notations, see  Remark~\ref{rem:p=0} below.
\end{rem}

\subsection{A crystal graph interpretation.}

For this subsection, we shall impose the stronger assumption that
{\em $p$ does not divide $|G|.$}

We assume in addition that $p>0$ except that in
Remark~\ref{rem:p=0} below we deal with the remaining case for
$p=0$.

Denote by $K(\mathcal{A})$ the Grothendieck group of a module
category $\mathcal A$ and by $\text{Irr}(\mathcal{A})$ the set of
pairwise non-isomorphic simple objects in $\mathcal A$. For
$\mu\in P_+$, let
$$
K(\mu)=\bigoplus_{n\geq 0}K\left(\mathcal{H}_n^{\mu}\text{-{\bf
mod}}\right), \qquad
K(\mu)_{\mathbb{C}}=\mathbb{C}\otimes_{\mathbb{Z}}K(\mu).
$$
Besides the functors $e^{\mu}_i$ and $f^{\mu}_i$ (cf. Remark
\ref{coincidence between cyclotomics}), we recall two additional
operators $\tilde{e}_i^{\mu}$ and $\tilde{f}_i^{\mu}$ on
$\coprod_{n\geq 0}\text{Irr}(\mathcal{H}_n^{\mu}\text{-{\bf
mod}})$ by letting $\tilde{e}_i^{\mu}L=\text{soc}(e^{\mu}_iL)$ and
$\tilde{f}_i^{\mu}L=\text{head}(f^{\mu}_iL)$ for each simple
$\mathcal{H}_n^{\mu}$-module $L$, cf. \cite[Section 8.2]{K2}.

Denote by $L(\mu)$ the irreducible highest weight
$\widehat{sl}_p$-module of highest weight $\mu\in P_+$. The
following is a degenerate counterpart of \cite{LLT, Ar1, Gro}.

\begin{prop}\cite[Theorem 9.5.1]{K2}\label{identification 1}
Let $\mu\in P_+$. Then $K(\mu)_{\mathbb{C}}$ is an
$\widehat{sl}_p$-module with the Chevalley generators acting as
$e^{\mu}_i$ and $f^{\mu}_i$ $(i\in \mathbb I)$;
as $\widehat{sl}_p$-modules, $K(\mu)_{\mathbb{C}} \cong L(\mu).$

Moreover, $\coprod_{n\geq
0}\text{Irr}(\mathcal{H}_n^{\mu}\text{-{\bf mod}})$ is isomorphic
to the crystal basis $B(\mu)$ of $U_q(\widehat{sl}_p)$-module
$L(\mu)$ with operators $\tilde{e}_i^{\mu}$ and
$\tilde{f}_i^{\mu}$ identified as Kashiwara operators.
\end{prop}

For $\lambda\in \Delta$, let
$$
K_G(\lambda)=\bigoplus_{n\geq 0}
K\left(\mathcal{H}_n^{\lambda}(G)\text{-{\bf mod}}\right).
$$
The functors $e_{i,\chi^k}^{\lambda}$ and $f_{i,\chi^k}^{\lambda}$
for $i \in \mathbb I$ and $1\le k \le r$ induce linear operators
(denoted by the same notations) on $K_G(\lambda)_\C:= \C\otimes_\Z
K_G(\lambda)$. The category equivalence in Theorem~\ref{morita
cyclotomic} induces a canonical linear isomorphism
\begin{eqnarray}  \label{isom}
F^\lambda : K_G(\lambda) \stackrel{\cong}{\longrightarrow}
K(\lambda[1]) \otimes \cdots \otimes K(\lambda[r]).
\end{eqnarray}

We shall identify $\mathcal{H}^{\lambda}_n(G)$-{\bf mod} with a
full subcategory of $\mathcal{H}_n(G)$-{\bf mod}. By
Lemma~\ref{lem:Mackey}, the functor $e_{i,\chi^k}^{\lambda}$
corresponds via $F^{\lambda}$ to $e^{\lambda[k]}_i$ applied to the
$k$-th factor on the right-hand side of \eqref{isom}. By Frobenius
reciprocity, $f_{i,\chi^k}^{\lambda}$ is left adjoint to
$e_{i,\chi^k}^{\lambda}$ and $f^{\lambda[k]}_i$ is left adjoint to
$e^{\lambda[k]}_i$, hence $f_{i,\chi^k}^{\lambda}$ corresponds to
$f^{\lambda[k]}_i$ applied to the $k$-th factor on the right-hand
side of \eqref{isom}. With the identification of
$\mathcal{H}^{\lambda}_n(G)$-{\bf mod} with a full subcategory of
$\mathcal{H}_n(G)$-{\bf mod}, Theorem~\ref{rule:Hng} implies the
modular branching rules for $\mathcal{H}^{\lambda}_n(G)$.
Combining these with Theorem~ \ref{morita cyclotomic} and
Proposition \ref{identification 1} we have established the
following.

\begin{thm}\label{identification 2}
Let $p> 0$. Then $K_G(\lambda)_{\mathbb{C}}$ affords a simple
$\widehat{sl}_p^{\oplus r}$-module isomorphic to
$L(\lambda[1])\otimes\cdots\otimes L(\lambda[r])$ with the
Chevalley generators 
of the $k$th summand of $\widehat{sl}_p^{\oplus r}$ acting as
$e_{i,\chi^k}^{\lambda}$ and $f_{i,\chi^k}^{\lambda}$ $(i\in
\mathbb I)$, for $1\le k \le r$.

Moreover, $\coprod_{n\geq
0}\text{Irr}(\mathcal{H}_n^{\lambda}(G)\text{-{\bf mod}})$ (and
respectively, the modular branching graph given by
Theorem~\ref{rule:Hng}) is isomorphic to the crystal basis
$B(\lambda[1])\otimes \cdots\otimes B(\lambda[r])$ (and
respectively, the corresponding crystal graph) for the simple
$U_{q}(\widehat{sl}_p^{\oplus r})$-module
$L(\lambda[1])\otimes\cdots\otimes L(\lambda[r])$.
\end{thm}

\begin{rem} \label{rem:wreathcrystal}
By Remark \ref{minimal quotient},
$\mathcal{H}^{\Lambda_0}_n(G)\cong\mathbb{F}G_n$. Observe that
$\Lambda_0[k]=\Lambda_0$, the $0$th fundamental weight of
$\widehat{sl}_p$ for all $1\leq k\leq r$. By Theorem
\ref{identification 2},
$\bigoplus_{n}{\mathbb{C}}\otimes_{\mathbb{Z}}K(\mathbb{F}G_n\text{-{\bf
mod}})$ affords a simple $\widehat{sl}_p^{\oplus r}$-module
isomorphic to $L(\Lambda_0)\otimes\cdots\otimes L(\Lambda_0)$
(compare with Corollary~ \ref{conjugate class equation}).
Actually, such a statement holds without any assumption on $p$ if
we replace $\mathbb{F}G_n\text{-{\bf mod}}$ above by
$\mathbb{F}G_n\text{-{\bf mod}}^s$ (defined similarly as
$\HngMods$), and the modular branching rule for $\mathbb{F}G_n$ in
Proposition~\ref{rule:Gn} can be interpreted as the crystal graph
$B(\Lambda_0)\otimes\cdots\otimes B(\Lambda_0)$.
\end{rem}

\begin{rem}  \label{rem:p=0}
The case $p=0$ can be treated similarly with somewhat more
complicated notations, and so we will be sketchy. An analogue of
Theorem~ \ref{morita cyclotomic} holds with suitably modified
algebra $A_{n,r}^\lambda$. The modification can be easily made
precise by an examination of the proof of Theorem~ \ref{morita
cyclotomic}: the annihilation ideal of $\mathcal H_{n_k}$
(corresponding to the $k$th tensor factor of $A_{n,r}^\lambda$) is
generated by $\prod_{i\in\mathbb{I}}(y_{1}-i/c_k)^{\lambda_i}$.
Note that all $i/ c_k$ are not necessarily integers, and so some
$k$th tensor factor of $A_{n,r}^\lambda$ is possibly a {\em
non-integral} cyclotomic (quotient) Hecke algebra $N_{n_k}$ of
$\mathcal H_{n_k}$.
For each such $k$, dividing $\{i/ c_k \mid i \in \Z \}$ into
congruence classes modulo $\Z$ leads to a decomposition of
$N_{n_k}$ as a tensor product of integral cyclotomic Hecke
algebras correspond to these congruence classes. This will lead to
an analogous formulation of Theorem~\ref{identification 2} via the
infinite-rank affine algebra $\widehat{sl}_\infty$.
%
\end{rem}

\end{document}